\documentclass[12pt,a4paper]{amsart}
\usepackage[letterpaper, margin=0.8in]{geometry}
\usepackage{amsfonts,amssymb,amsthm,epsf, graphics, bbm, mathabx, verbatim, caption, subcaption, enumerate, amsmath}
\usepackage{comment}

\usepackage{indentfirst}

\usepackage{tikz}
\usetikzlibrary{patterns}

\newcommand{\ZZ}{\mathbb{Z}}

\newcommand{\RR}{\mathbb{R}}
\newcommand{\CC}{\mathbb{C}}
\newcommand{\N}{\mathbb{N}}

\newcommand{\norm}[1]{\left\lVert#1\right\rVert}

\newcommand{\eps}{\varepsilon}

\newtheorem{thm}{\bf Theorem}
\newtheorem{theorem}{Theorem}[section]

\newtheorem{proposition}[theorem]{Proposition}
\newtheorem{lem}[thm]{\bf Lemma}
\newtheorem{lemma}[theorem]{Lemma}

\theoremstyle{remark}
\newtheorem*{remark}{Remark}

\theoremstyle{definition}
\newtheorem{definition}[theorem]{Definition}

\usepackage{mathtools}

\title{Fourier dimension of constant rank hypersurfaces}
\author{Junjie Zhu}

\newcommand{\Addresses}{{
  \bigskip
  \footnotesize

  \textsc{Department of Mathematics, 1984 Mathematics Road, University of British Columbia, Vancouver, BC Canada V6T 1Z2}\par\nopagebreak
  \textit{E-mail address}: \texttt{jzhu@math.ubc.ca}
}}
\date{\today}


\begin{document}

\begin{abstract}
     Any hypersurface in $\RR^{d+1}$ has a Hausdorff dimension of $d$. However, the Fourier dimension depends on the finer geometric properties of the hypersurface. For example, the Fourier dimension of a hyperplane is 0, and the Fourier dimension of a hypersurface with non-vanishing Gaussian curvature is $d$. Recently, Harris showed that the Euclidean light cone in $\RR^{d+1}$ has a Fourier dimension of $d-1$, which leads one to conjecture that the Fourier dimension of a hypersurface equals the number of non-vanishing principal curvatures. We prove this conjecture for all constant rank hypersurfaces. Our method involves substantial generalizations of Harris's strategy.
\end{abstract}

\maketitle

\section{Introduction}

The decay rate of the Fourier transforms of measures supported on manifolds has been of central interest in harmonic analysis. The Fourier transform of the normalized surface measure $\sigma$ on the unit sphere $\mathbb{S}^{d} \subset \RR^{d+1}$ is given by
$$\widehat{\sigma}(\xi) = c_d |\xi|^{\frac{1-d}{2}}J_{\frac{d-1}{2}}(2\pi |\xi|),$$
where $J_{\frac{d-1}{2}}$ is the Bessel function of integral order $\frac{d-1}{2}$, and $c_d$ is the normalizing constant \cite{Stein1993}. It satisfies 
\begin{equation}
\label{e_sphere}
 |\widehat{\sigma}(\xi)| \leq C_{d} |\xi|^{-\frac{d}{2}}   \text{ for a } C_d > 0.
\end{equation}

Describing sets in $\RR^{d+1}$, which may not be manifolds, using the Fourier decay of measures supported on them is one motivation for studying the \textit{Fourier dimension}. Let $\mathcal{M}(A)$ be the set of measures $\mu$ supported on $A$ with finite total mass $\norm{\mu}_{1} := \mu(A) < \infty$. The \textit{Fourier transform} of a measure $\mu$ at $\xi \in \RR^{n}$ is defined as $\widehat{\mu}(\xi) := \int e^{-2\pi i x \xi} d\mu(x)$.
The \textit{Fourier dimension} for a Borel $A \subset \RR^{d+1}$ is defined as 
$$\dim_F(A) := \sup \{s \in [0, d+1]: \exists \mu \in \mathcal{M}(A), \sup_{\xi \in \RR^{d+1}} |\xi|^{\frac{s}{2}}|\widehat{\mu}(\xi)| < \infty  \}.$$

The notion of the Fourier dimension is ubiquitous in harmonic analysis and geometric measure theory. It provides a lower bound on the \textit{Hausdorff dimension}, which is discussed further in Section \ref{hf_dim}.
Studies of Fourier dimensions include probabilistic \cite{Salem, Kahane, Bluhm, LP} and deterministic constructions \cite{Kaufman, Fraser_Hambrook} of \textit{Salem sets} with equal Fourier and Hausdorff dimensions and properties satisfied by any Borel sets of large Fourier dimensions \cite{liang}.

The objective of this paper is to study the case where $A$ is a constant rank hypersurface and compute its Fourier dimension.

\subsection{Constant rank  hypersurfaces}

Let $M\subset \RR^{d+1} $ be an orientable smooth hypersurface, which is a topological manifold of dimension $d$ with a normal direction $N: M \to \mathbb{S}^{d}$. A way to describe $M$ is through notions of curvatures, defined through the eigenvalues of the \textit{Weingarten map} \cite{edg}. 
Let $T_pM$ be the tangent space of $M$ at $p \in M$. The \textit{Weingarten map} $L_p: T_p M \to T_p M$ at $p \in M$ is the linear map $L_p(v) = -D_v N = -\frac{d}{dt} (N \circ \gamma) (0)$, where $\gamma: I \to M$ is a curve with $\gamma(0)=p$, $\gamma'(0)=v$. The \textit{principal curvatures} of $M$ at $p$ are the eigenvalues of the map $L_p$, and the \textit{Gaussian curvature} of $M$ at $p$ is the product of the eigenvalues, which equals the determinant of $L_p$. We note that Gaussian and principal curvatures at $p \in M$ are independent of the parametrization of $M$ and the choice of a basis for $T_p M$. The Weingarten map is self-adjoint, so the eigenvalues of $L_p$ are real.

In this manuscript, we focus on \textit{hypersurfaces of constant rank} $k$ where at all points $p \in M$, $k$ principal curvatures are non-zero, and $d-k$ principal curvatures are zero. For example, in addition to cones and cylinders, tangent surfaces are hypersurfaces of constant rank $1$ in $\RR^3$. Two tangent surfaces are shown in Section \ref{sec_ub_fd_ex}. 
A hypersurface in higher dimensions not classified as a cone or a cylinder is $$\left\{\left.\gamma(t)+\sum_{j=1}^{d-1} v_j \frac{d^j \gamma}{dt^j}(t) \right\rvert t, v_j \in \RR\right\} \subset \RR^{d+1},$$
where $\gamma(t) = (t, t^2, \cdots, t^d, t^{d+1})$. It is of constant rank $1$ that generalizes tangent surfaces in $\RR^3$ to higher dimensions. Another hypersurface of constant rank $2$ in $\RR^{4}$ is 
$$\left\{ (t, \sin(t)e^{-s}, \sin(2t)e^{-4s}, \sin(3t)e^{-9s})+ v  (1, \cos(t)e^{-s}, 2\cos(2t)e^{-4s}, 3\cos(3t)e^{-9s}) | t, s, v \in \RR\right\}.$$
Some more examples are cylinders of cones $\{(kv, k, h) | k, h \in \RR, v \in S \} \subset \RR^{d+1}$, where $S$ is a hypersurface in $\RR^{d-1}$ with non-vanishing Gaussian curvature.

Although \textit{surfaces of constant rank} have been studied extensively in differential geometry \cite{Ushakov}, they are less examined in geometric measure theory, where geometric objects are analyzed via measures supported on them.

\subsection{Main Result}

This manuscript focuses on computing $\dim_F(M)$, interpreting the properties of $M$ that $\dim_F(M)$ captures, and differentiating between topological, Hausdorff, and Fourier dimensions on $M$. Since the Hausdorff dimension is preserved under diffeomorphisms, which are bi-Lipschitz, $\dim_H(M) = d$ \cite[Proposition 3.1]{falconer}. 

By (\ref{e_sphere}), the unit sphere $\mathbb{S}^{d}$ has Fourier dimension $d$. Hlawka \cite{Hlawka} showed that hypersurfaces with non-vanishing Gaussian curvature also have a Fourier dimension of $d$, which suggests that the Fourier dimension does not depend on the symmetry of the hypersurface. Harris \cite{Fraser_2022} showed that the Euclidean light cone in $\RR^{d+1}$, which is a hypersurface with zero Gaussian curvature but of constant rank $d-1$, has a Fourier dimension of $d-1$. The proof relies on the rotational symmetry of the light cone. The author's previous work \cite{jz} adapts the proof of \cite{Fraser_2022} and shows that cones and cylinders of rank $d-1$ without rotational symmetry also have a Fourier dimension of $d-1$. Table \ref{table:1} summarizes these results.

\begin{table}[h!]
\centering
\begin{tabular}{c | c| c| c} 
 
 Hypersurface $M$ & Constant Rank & $\dim_F(M)$ & Reference \\ [0.5ex] 
 \hline
 non-vanishing Gaussian curvature ($\mathbb{S}^{d}$, etc) & $d$ & $d$ & \cite{Hlawka} \\ [0.5ex] 
 \hline
 Light cone $ \{ h(x, 1) | x \in \mathbb{S}^{d-1}, h \in \RR \}$ & $d-1$ & $d-1$ & \cite{Fraser_2022} \\ [0.5ex]
 \hline
 Generalized cone $ \{h(x, 1) |x \in S, h \in \RR \}$ & $d-1$ & $d-1$ & \cite{jz} \\ [0.5ex]
 \hline
 $(d-1)$-cylinder $S \times \RR$ & $d-1$ & $d-1$ & \cite{jz} \\
 \hline
 Hyperplane $\{x: x_{d+1}=0\}$& $0$ & $0$ &  \\
\end{tabular}
\caption{Fourier dimension of some hypersurfaces, where $S \subset \RR^{d}$ is a hypersurface with non-zero Gaussian curvature.}
\label{table:1}
\end{table}

Previous results suggest that the Fourier dimension of $M$ equals its rank. In this manuscript, we confirm that such a hypothesis holds

\begin{theorem}
\label{t_main}
    Let $M$ be a smooth hypersurface of constant rank $k$ in $\RR^{d+1}$, then $\dim_F(M) = k$.
\end{theorem}

\begin{remark}
    \begin{enumerate}[a.]
        \item Theorem \ref{t_main} shows a direct link between the principal curvatures and the Fourier dimension for constant rank hypersurfaces. We can view the Fourier dimension as a generalization of the number of non-zero principal curvatures defined on hypersurfaces. It is a generalization of the results in \cite{Fraser_2022, jz}. It includes all hypersurfaces of constant rank $k$ for $k < d-1$ and hypersurfaces of constant rank $d-1$ that are neither cones nor cylinders. 
        \item Previously known bounds on the Fourier dimension of a constant rank $k$ hypersurface $M \subset \RR^{d+1}$ are
        $$k \leq \dim_F(M) \leq k\left(\frac{d+1}{k+1}\right).$$
          The lower bound from \cite{Littman} arises from the Fourier transform of measures supported on the hypersurface induced by the Lebesgue measure. The upper bound is obtained from the study of the $p$-thin problem \cite{gkh}. If $\mu \in \mathcal{M}(M)$ satisfies $|\widehat{\mu}(\xi)| \leq C(1+|\xi|)^{-\frac{\alpha}{2}}$ for $C, \alpha > 0$, $\widehat{\mu} \in L^{p}(\RR^{d+1})$ for all $p > \frac{2(d+1)}{\alpha}$. Such $p$ must satisfy $p > \frac{2(k+1)}{k}$, so $\alpha \leq k\left(\frac{d+1}{k+1}\right)$.
    \end{enumerate}
\end{remark}

\subsection{Overview of the proof}

To establish Theorem \ref{t_main}, we need to show that the Fourier dimension of a hypersurface $M$ of constant rank $k$ is bounded below and above by $k$. For the upper bound (Proposition \ref{p_fd_leq_k}), the main strategy adapted from \cite{Fraser_2022} is to create a new measure by averaging a series of push-forward of an old measure supported on $M$.

An adaptation of the methodology from \cite{Fraser_2022} is needed. In previous works \cite{Fraser_2022, jz}, it was believed that the specific maps used to push-forward measures are required to map $M$ to itself, so the new measure is still supported on the hypersurface. While such assumptions pose no issue for cones and cylinders, similar maps may not exist for more general hypersurfaces for us to apply the same machinery. 

Fortunately, we note that the new measure does not have to be supported on the original $M$. We propose new maps that push forward an old measure to a series of similar hypersurfaces, and we still create a new measure as the average of the series of measures. The proof in Section \ref{sec764} shows that such a new measure may not be supported on $M$, but it has a Fourier decay exponent in one direction that is not less than that of the old measure. Moreover, the Fourier transform of the new measure is related to the second fundamental form of the hypersurface $M$, which prevents the Fourier decay exponent of the new measure from being greater than $k/2$. 

\subsection{Proof of the main result and organization of the manuscript}

Theorem \ref{t_main} follows from Theorem \ref{t33} and Proposition \ref{p_fd_leq_k}.

\begin{thm}[\cite{Littman}]
\label{t33}
    Suppose that at least $k$ of the principal curvatures are not zero at all points in $M$. Then, there exists a measure $\mu \in \mathcal{M}(M)$ such that $\sup_{\xi \in \RR^{d+1}}|\xi|^{\frac{k}{2}}|\widehat{\mu}(\xi)| < \infty$.
\end{thm}

Therefore, if $M$ is a constant rank $k$ hypersurface, $\dim_F(M) \geq k$. It remains to prove the following.

\begin{proposition}
\label{p_fd_leq_k}
    $\dim_F(M) \leq k$.
\end{proposition}

We present examples of Proposition \ref{p_fd_leq_k} in Section \ref{sec_ub_fd_ex} and prove Proposition \ref{p_fd_leq_k} in Section \ref{sec764}. 

\subsection{Notations}

For two functions $f, g: D \to \RR_{\geq 0}$ with a domain $D$, we write $f \lesssim g$ to denote that there exists a $c>0$, such that for all $x\in D$, $f(x) \leq c g(x)$.

For $\phi: \RR^{n} \to \RR^{n}$, we denote the Jacobian of $\phi$ as ${\bf J} \phi$.

If $(X, \mathcal{X})$ and $(Y, \mathcal{Y})$ are measurable spaces and the mapping $F: X \to Y$ is $(\mathcal{X}, \mathcal{Y})$-measurable, the push-forward map $F^{\#}$ is defined to be $F^{\#}\mu(U) = \mu(F^{-1}(U))$ for any set $U \in \mathcal{Y}$ and any measure $\mu$ on $(X, \mathcal{X})$.

On $(X, \mathcal{X})$, if $f$ is a measurable function and $\mu$ is a measure, we denote $\langle f, \mu \rangle = \int_{X} f d\mu$.

\subsection{Acknowledgments}
I thank Malabika Pramanik and Jialing Zhang for their discussions on this project. I thank Yuveshen Moorogen for editing suggestions for an earlier version of the manuscript.

\section{Upper bounds on the Fourier dimension: Examples}
\label{sec_ub_fd_ex}

In this section, we demonstrate the process of showing that the Fourier dimensions of two tangent surfaces, which are hypersurfaces of constant rank $1$ in $\RR^{3}$ different from generalized cones and cylinders, are at most 1. We note that the normal vectors used in the section are not unit vectors, but the conclusion is not affected.

\subsection{Tangent surface generated by a Helix}
\label{sec_ahe}

Let $\gamma: \RR \to \RR^3$, $\gamma(t)=(t, t^2, t^3)$, and 
\begin{equation}
\label{eq31_s}
    S = \{\Phi(t, v) : = \gamma(t)+ v\gamma'(t) | t \in \RR, v \neq 0  \}
\end{equation}
 be the tangent surface generated by $\gamma$,
where $\gamma'(t) = (1, 2t, 3t^2)$. $\gamma''(t) = (0, 2, 6t)$, and a normal vector of $S$ at a point $\Phi(t, v)$ is 
\begin{equation}
\label{eq31_n}
     \overrightarrow{n}(t) := (3t^2, -3t, 1). 
\end{equation}
The tangent surface $S$ is of rank 1 since the second fundamental form at $\Phi(t, v)$ is $$\begin{pmatrix} \Phi_{tt} \cdot  \overrightarrow{n} &  \Phi_{tv} \cdot  \overrightarrow{n} \\\Phi_{vt} \cdot  \overrightarrow{n} &  \Phi_{vv} \cdot  \overrightarrow{n}\end{pmatrix}=\begin{pmatrix} 6v &  0 \\0 & 0\end{pmatrix}$$ has rank $1$ when $v \neq 0$.

\begin{figure}[h]
  \centering
   \includegraphics[width=119mm]{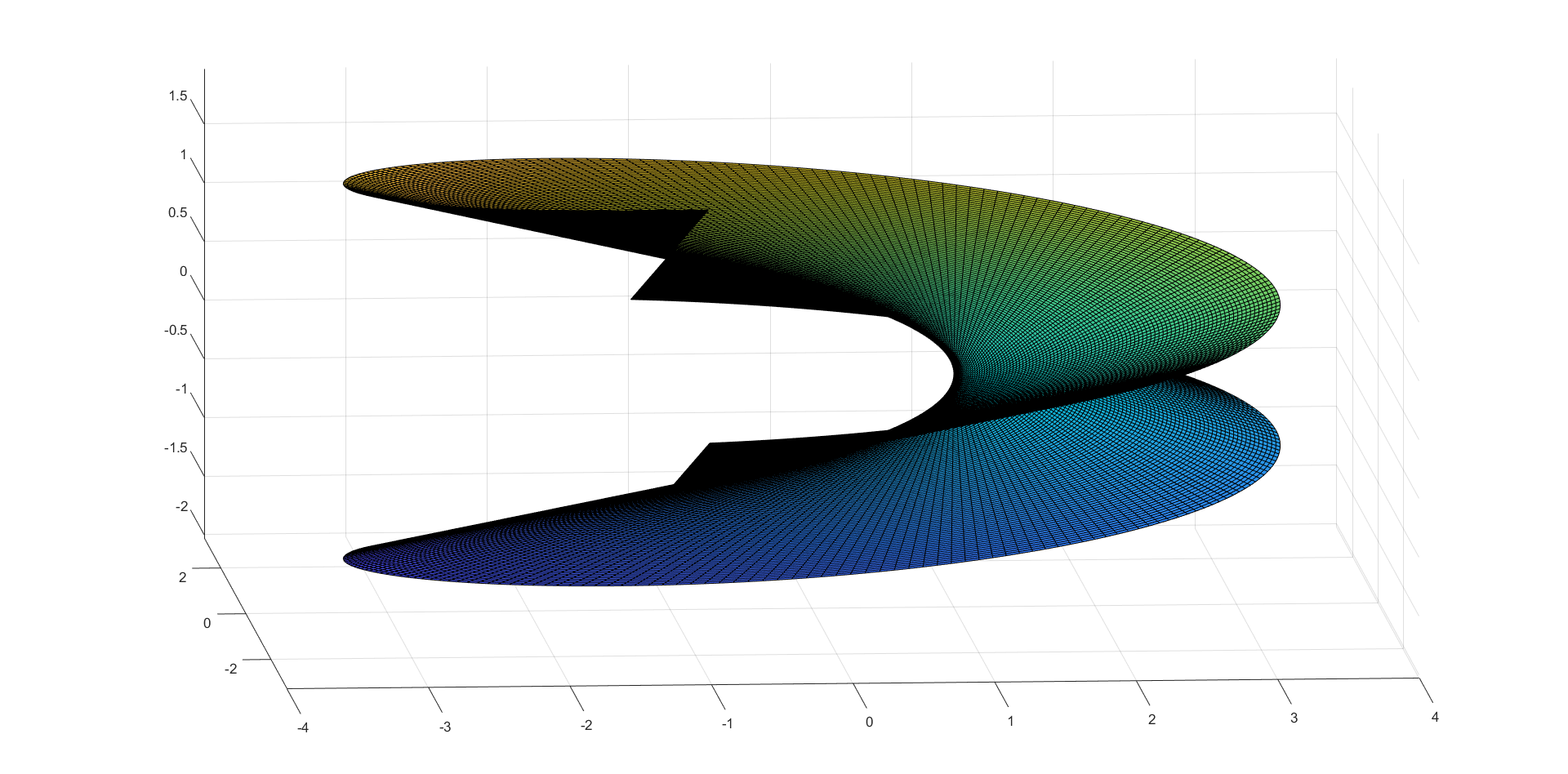}
 \caption{A 3D plot of $S$ from Matlab.}
\end{figure}


\begin{proposition}
\label{p_ahe_fd1}
    $\dim_F(S) \leq 1$.
\end{proposition}

\begin{proof}
    Let $\mu \in \mathcal{M}(S)$. Suppose that there exists $\beta > 0$, such that $|\widehat{\mu}(\xi)| \leq C |\xi|^{-\frac{\beta}{2}}$ for $\xi \in \RR^3$, $C > 0$. By Lemma \ref{lem_reduct}, we may assume that 
    $S = \{\Phi(t, v) | |t| \leq c, v \in [a, b]\}$
    for $0 < a < b$, $c > 0$.
     
    We construct new measures. Let the reparametrization $T_s: S \to S$ be 
    \begin{equation}
    \label{eq31_TS}
        T_s(\Phi(t, v)) = \Phi(t-s, v).
    \end{equation}
    
    \begin{figure}[!h]
\centering
\beginpgfgraphicnamed{Illustration}
\begin{tikzpicture}
\draw[gray, thick] (0, 0) -- (3,0.5);
\draw[gray, thick] (1, 0.4) -- (3, 2.2);
\draw[gray, thick] (1.2, 1.5) -- (0.8, 3);
\filldraw[black] (2, 1.3) circle (2pt) node[anchor=west]{$\Phi(t, v)$};
\draw (0,0) .. controls (1, 0.2) and (2, 0.6) .. (0.5 ,3);

\draw[gray, thick] (10, 0) -- (13,0.5);
\draw[gray, thick] (11, 0.4) -- (13, 2.2);
\draw[gray, thick] (11.2, 1.5) -- (10.8, 3);
\filldraw[black] (12, 0.3) circle (2pt) node[anchor=north]{$\Phi(t-s, v)$};
\draw (10,0) .. controls (11, 0.2) and (12, 0.6) .. (10.5 ,3);
\draw[->]  (4, 1) -- (8, 1)
node [above,text width=3cm,text centered,midway]
{Re-parametrize by $T_s$
};

\draw[black, fill=gray, fill opacity=0.3] (2, 1.3) -- (2.4, 1.7) -- (2.4,2) -- (2,1.6) -- cycle;

\draw[black, fill=black, fill opacity=0.6] (12, 0.3) -- (12.4, 0.4) -- (12.4, 0.7) -- (12, 0.6) -- cycle;
\end{tikzpicture}
\endpgfgraphicnamed
\caption{Illustration of the re-parametrization by $T_s$ on $S$.}
\end{figure}
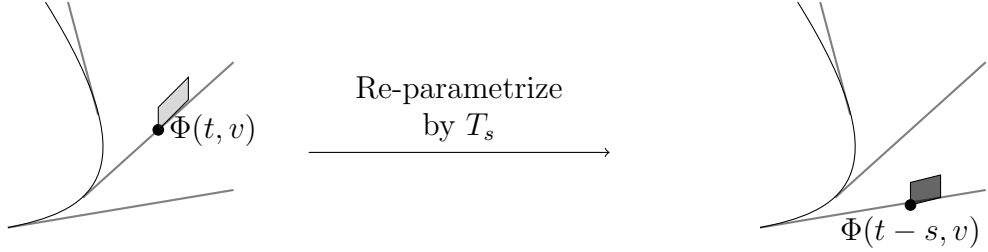

Let $\mu_s \in \mathcal{M}(S)$ be 
\begin{equation}
\label{eq31_mus}
    \mu_s = T_s^{\#} \mu.
\end{equation}
Let $\psi \in C_{0}^{\infty}(\RR)$ be non-negative bump function with $\psi(s)=1$ when $|s| \leq c$, $\psi(s)=0$ when $|s| \geq 2c$. Then we let the average measure $\nu$ be defined as
\begin{equation}
    \label{eq31_nu}
    \int f d\nu = \int \langle f, \mu_s \rangle \psi(s) ds
\end{equation}
for non-negative Borel functions $f$.

Now, we examine $\widehat{\nu}$ in the $e_3$ direction, where $e_3$ is also a normal vector at $\gamma(0, v)$. Proposition \ref{p_ahe_fd1} is a consequence of the following two lemmas.

    \begin{lemma}
        \label{lem_ahe_beta}
        For $\rho > 0$, $|\widehat{\nu}(\rho e_3)| \lesssim \rho^{-\frac{\beta}{2}}$.
    \end{lemma}

    \begin{lemma}
        \label{lem_ahe_half}
        There exists $\rho_0 > 0$, such that for $\rho > \rho_0$,
        $|\widehat{\nu}(\rho e_3)| \gtrsim \rho^{-\frac{1}{2}}$. \qedhere
    \end{lemma}
\end{proof}

\begin{proof}[Proof of Lemma \ref{lem_ahe_beta}]
    First, we show that for each $s\in \RR$, $|\widehat{\mu_s}(\rho e_3)| \lesssim \rho^{-\frac{\beta}{2}}$. By unwinding the definitions of push-forward measures, Fourier transform, and integration on the manifold, we have
    \begin{equation}
    \label{e_helix_mus}
        \begin{aligned}
            \widehat{\mu_s}(\rho e_3) & = \int e^{-2\pi i x \cdot \rho e_3} dT_s^{\#}\mu(x) & \text{ by (\ref{eq31_mus})}\\
            & = \int e^{-2\pi i \rho T_s(\Phi(t, v)) \cdot  e_3} d\mu\circ\Phi(t, v) \\
            & = \int e^{-2\pi i \rho \Phi(t-s, v) \cdot e_3} d\mu\circ \Phi(t, v) & \text{ by (\ref{eq31_TS})}.
        \end{aligned}
    \end{equation}

    From (\ref{eq31_s}) and (\ref{eq31_n}),
    \begin{equation}
    \label{eq_helix_Phie3}
             \Phi(t-s, v) \cdot e_3 
             = \Phi(t, v)\cdot \overrightarrow{n}(s) - s^3.
    \end{equation}
    Then, (\ref{e_helix_mus}) becomes
    \begin{equation}
        \begin{aligned}
            \widehat{\mu_s}(\rho e_3)
            & = \int e^{-2\pi i \rho \Phi(t-s, v) \cdot e_3} d\mu\circ \Phi(t, v) \\
            & = \int e^{-2\pi i \rho [\Phi(t, v)\cdot \overrightarrow{n}(s) -s^3]} d\mu\circ \Phi(t, v) & \text{ by (\ref{eq_helix_Phie3})}\\
            & = e^{\pi i \rho s^3} \widehat{\mu}(\rho \overrightarrow{n}(s)).
        \end{aligned}
        \tag{\text{(\ref{e_helix_mus}) c.}}
    \end{equation}
    Since $|\overrightarrow{n}(s)| \geq 1$, $|\widehat{\mu_s}(\rho e_3)| = |\widehat{\mu}(\rho \overrightarrow{n}(s))| \leq C |\rho \overrightarrow{n}(s)|^{-\frac{\beta}{2}} \leq C \rho^{-\frac{\beta}{2}}.$ Then,
    \begin{align*}
        |\widehat{\nu}(\rho e_3)| \leq & \int |\widehat{\mu_s}(\rho \overrightarrow{n}(s))| \psi(s) ds & \text{ by (\ref{eq31_nu})}\\
        \leq & C \norm{\psi}_{L^{1}} \rho^{-\frac{\beta}{2}}.  & &\qedhere
    \end{align*}
\end{proof}

\begin{proof}[Proof of Lemma \ref{lem_ahe_half}]
An expression of $\widehat{\nu}(\rho e_3)$ is given by
\begin{equation}
\label{e_helix_v_exp}
    \begin{aligned}
        \widehat{\nu}(\rho e_3) & = \int \int e^{-2 \pi i \rho \Phi(t-s, v) \cdot e_3} d\mu \circ \Phi(t, v) \psi(s) ds & \text{ by (\ref{eq31_nu})} \\
        & = \int \int e^{-2 \pi i \rho \Phi(t-s, v) \cdot e_3}  \psi(s) ds d\mu \circ \Phi(t, v).
    \end{aligned}
\end{equation}

We apply the stationary phase method to study the inner integral
\begin{equation*}
   I(\rho; t, v) := \int e^{-2 \pi i \rho \Phi(t-s, v) \cdot e_3}  \psi(s) ds.
\end{equation*}

Fixing $t, v$, we define the phase function $\phi$ of $I(\rho; t, v)$ as
$$\phi(s) := \Phi(t-s, v) \cdot e_3 = (t-s)^3+3v(t-s)^2 \text { from } (\ref{eq_helix_Phie3}).$$  Its derivatives are
$$\phi'(s) = -3(t-s)^2-6v(t-s),$$ and
$$\phi''(s) = 6(t-s)+6v.$$ The solutions to $\phi'(s)=0$ are
$$s = t, t+2v.$$
Since $v \in [a, b]$, if $-2c+2a>2c$ or $2c< a$, the only critical point in $\text{spt } \psi$ is $s=t$. Note that $$\phi(t)=0, \phi'(t)=0, \phi''(t) = 6v.$$ Then, by the stationary phase (part b, Theorem \ref{t_asym}), there exists $D_0 > 0$ independent of $\rho$, $t$, and $v$, such that when $\rho v$ is sufficiently large,
$$\left|I(\rho; t, v) - \left(\frac{-i}{6 \rho v}\right)^{\frac{1}{2}}\psi(t)\right| \leq D_0(\rho v)^{-1}.$$

Since $\psi(t) = 1$ for $|t| \leq c$, from (\ref{e_helix_v_exp}),
\begin{equation*}
    \begin{split}
        \left|\widehat{\nu}(\rho e_3) - \int \left(\frac{-i}{6 \rho v}\right)^{\frac{1}{2}}  d\mu \circ \Phi(t, v) \right| & \leq   \int D_0(\rho v)^{-1} d\mu \circ \Phi(t, v). \\
    \end{split}
\end{equation*}

As $v \in [a, b]$, 
\begin{align*}
    \left|\int \left(\frac{-i}{6 \rho v}\right)^{\frac{1}{2}}  d\mu \circ \Phi(t, v) \right| \geq & 6^{-\frac{1}{2}} b^{-\frac{1}{2}}\rho^{-\frac{1}{2}}, \\
    \left|\int D_0(\rho v)^{-1} d\mu \circ \Phi(t, v)\right| \leq & D_0 a^{-1}\rho^{-1}.
\end{align*}
Therefore,
$|\widehat{\nu}(\rho e_3)|$ is bounded below by $4^{-1}b^{-\frac{1}{2}} \rho^{-\frac{1}{2}}$ for $\rho$ sufficiently large.
\end{proof}

\subsection{A tangent surface generated by a perturbed helix}
\label{sec_phelix}

In this subsection, we examine the tangent surface generated by the curve $\alpha(t) = (t, t^2+t^4, t^3)$. 

\begin{proposition}
\label{p_phelix_fd_1}
    Let $I = [-2c, 2c]$ for $c>0$, and let $\alpha: I \to \RR^3$ be defined above. Let $$S := \{\alpha(t) + v\alpha'(t) | t \in I, a < v < b\}$$ be the tangent surface generated by $\alpha$, where $0 < a < b$. Then, $\dim_F(S) \leq 1$.
\end{proposition}

We denote
$\Phi: I \times (a, b) \to \RR^3$ be $\Phi(t, v)=\alpha(t) + v \alpha'(t)$.
Proposition \ref{p_phelix_fd_1} holds if the following lemma as an analog of (\ref{eq_helix_Phie3}) holds.

\begin{lemma}
\label{p_phelix_Ts}
For $s \in I$, there exists $T_s: S \to \RR^{3}$, such that
    \begin{equation}
    \label{p_phelix_ex}
     T_s(\alpha(t)+v\alpha'(t)) \cdot e_3  = (\alpha(t)+v\alpha'(t) - \alpha(s)) \cdot \overrightarrow{n}(s),
    \end{equation}
    where 
    \begin{equation}
    \label{N_phelix}
        \overrightarrow{n}(s) = (-3s^2(2s^2-1), -3s, 6s^2+1)
    \end{equation}
    is a normal vector of $S$ at $\Phi(s, v)$ with $\overrightarrow{n}(0) = e_3$.

    
\end{lemma}

\begin{proof}[Proof of Proposition \ref{p_phelix_fd_1} assuming Lemma \ref{p_phelix_Ts}]

Let $\mu \in \mathcal{M}(S)$, and suppose that there exists  $\beta > 0$, such that $|\widehat{\mu}(\xi)| \leq C |\xi|^{-\frac{\beta}{2}}$ for $\xi \in \RR^3$, $C > 0$. 
    
We construct new measures. Let $\mu_s \in \mathcal{M}(\RR^3)$ be
\begin{equation}
\label{eq32_mus}
\mu_s = T_s^{\#} \mu.
\end{equation} Let $\psi \in C_{0}^{\infty}(\RR)$ be a non-negative bump function with $\psi(s)=1$ when $|s| \leq c$, $\psi(s)=0$ when $|s| \geq 2c$. Then we let the average measure $\nu$ be defined as
\begin{equation}
    \label{eq32_nu}
    \int f d\nu = \int \langle f, \mu_s \rangle \psi(s) ds
\end{equation}
for non-negative Borel functions $f$. Proposition \ref{p_phelix_fd_1} is a consequence of the following two lemmas.

    \begin{lemma}
        \label{l_phelix_beta}
        For $\rho > 0$, $|\widehat{\nu}(\rho e_3)| \lesssim \rho^{-\frac{\beta}{2}}$.
    \end{lemma}

    \begin{lemma}
        \label{l_phelix_half}
        There exists $\rho_0 > 0$, such that for $\rho > \rho_0$,
        $|\widehat{\nu}(\rho e_3)| \gtrsim \rho^{-\frac{1}{2}}$. \qedhere
    \end{lemma}
\end{proof}

Compared with previous methods, the computations in Lemmas \ref{l_phelix_beta} and \ref{l_phelix_half} do not require $\text{spt }\nu \subset S$.

\begin{proof}[Proof of Lemma \ref{l_phelix_beta}]
    First, we show that for each $s\in I$, $|\widehat{\mu_s}(\rho e_3)| \lesssim \rho^{-\frac{\beta}{2}}$. By unwinding the definitions of the push-forward measure, Fourier transform, and integration on the manifold, we have
    \begin{equation}
    \label{eq_phelix_mus_ex}
        \begin{aligned}
            \widehat{\mu_s}(\rho e_3) & = \int e^{-2\pi i y \cdot \rho e_3} dT_s^{\#}\mu(y) & \text{ by (\ref{eq32_mus})} \\
            & = \int e^{-2\pi i \rho T_s(\Phi(t, v)) \cdot  e_3} d\mu\circ\Phi(t, v) \\
            & = \int e^{-2\pi i \rho [(\alpha(t)+v\alpha'(t))  \cdot \overrightarrow{n}(s) - \alpha(s) \cdot \overrightarrow{n}(s)]} d\mu\circ \Phi(t, v) & \text{ by (\ref{p_phelix_ex})} \\
            & = e^{2 \pi i \rho \alpha(s) \cdot \overrightarrow{n}(s)} \widehat{\mu}(\rho \overrightarrow{n}(s)).
        \end{aligned}
    \end{equation}
    Since $|\overrightarrow{n}(s)| \geq 1$,  $|\widehat{\mu_s}(\rho e_3)| = |\widehat{\mu}(\rho \overrightarrow{n}(s))| \leq C |\rho \overrightarrow{n}(s)|^{-\frac{\beta}{2}} \leq C \rho^{-\frac{\beta}{2}}.$ Then,
        \begin{align*}
            |\widehat{\nu}(\rho e_3)| \leq & \int |\widehat{\mu_s}(\rho e_3)| \psi(s) ds & \text{ by (\ref{eq32_nu})}\\
            \leq & C \norm{\psi}_{L^{1}} \rho^{-\frac{\beta}{2}}. & & \qedhere
        \end{align*} 
\end{proof}

\begin{proof}[Proof of Lemma \ref{l_phelix_half}]
An expression of $\widehat{\nu}(\rho e_3)$ is given by
\begin{equation}
\label{eq_phelix_nu_ex}
    \begin{aligned}
        \widehat{\nu}(\rho e_3) & = \int \int e^{-2\pi i \rho T_s(\Phi(t, v)) \cdot  e_3} d\mu\circ\Phi(t, v) \psi(s) ds & \text{ by (\ref{eq32_nu})}\\
        & = \int \int e^{-2\pi i \rho T_s(\Phi(t, v)) \cdot  e_3}  \psi(s) ds d\mu \circ \Phi(t, v).
    \end{aligned}
\end{equation}

We apply the stationary phase method to study the inner integral
\begin{equation*}
   I(\rho; t, v) := \int e^{-2\pi i \rho T_s(\Phi(t, v)) \cdot  e_3}  \psi(s) ds.
\end{equation*}

Fixing $t, v$, we define the phase function $\phi$ of $I(\rho; t, v)$ as
    \begin{align*}
        \phi(s) & := T_s(\Phi(t, v)) \cdot  e_3 = (\alpha(t)+v\alpha'(t))  \cdot \overrightarrow{n}(s) - \alpha(s) \cdot \overrightarrow{n}(s) & \text{ by }(\ref{p_phelix_ex}) \\
        & = [(t-s, t^2+t^4-s^2-s^4, t^3-s^3) + v(1, 2t+4t^3, 3t^2)] \cdot (
        -6s^4+3s^2, -3s, 6s^2+1)
    \end{align*}
Its first derivative is
\begin{equation*}
    \begin{split}
    \phi'(s) 
        = & -3(t-s) [(t-s)^3+4v(t-s)^2+(1-6s^2)(t-s)+2v(1-6s^2)].
    \end{split}
\end{equation*}
One critical point is $s=t$. If $c$ is sufficiently small, the only critical point in $\text{spt }\psi$ is $s=t$. Note that 
$$\phi(t)=0, \phi'(t)=0, \phi''(t) = v(6-36t^2).$$ 
Then, by the stationary phase (part b of Theorem \ref{t_asym}),
$$\left| I(\rho; t, v) - \left(\frac{-i}{\rho  v  (6-36t^2)}\right)^{\frac{1}{2}}\psi(t)\right| \leq D_1 (\rho v)^{-\frac{3}{2}}$$
for a $D_1 > 0$ independent of $\rho$, $t$, and $v$. Since $\psi(t) = 1$ for $|t| \leq c$, continuing with (\ref{eq_phelix_nu_ex}),
\begin{equation*}
    \begin{split}
        \left| \widehat{\nu}(\rho e_3) - \int \left(\frac{-i}{\rho  v  (6-36t^2)}\right)^{\frac{1}{2}}  d\mu \circ \Phi(t, v)\right| & \leq  \int D_1 (\rho v)^{-1}  d\mu \circ \Phi(t, v)\\
    \end{split}
\end{equation*}

As $v \in [a, b]$, 
\begin{align*}
  \left|\int \left(\frac{-i}{\rho  v  (6-36t^2)}\right)^{\frac{1}{2}}  d\mu \circ \Phi(t, v)\right| \geq & 6^{-\frac{1}{2}}b^{-\frac{1}{2}}\rho^{-\frac{1}{2}}, \\
  \left|\int D_1 (\rho v)^{-1}  d\mu \circ \Phi(t, v)\right| \leq & D_1 a^{-1} \rho^{-1}.
\end{align*}
Therefore, $|\widehat{\nu}(\rho e_3)|$ is bounded below by $4^{-1}b^{-\frac{1}{2}} \rho^{-\frac{1}{2}}$ for $\rho$ sufficiently large.
\end{proof}

Now, we show the existence of $T_s$ for the perturbed helix.

\begin{proof}[Proof of Lemma \ref{p_phelix_Ts}]
    To satisfy (\ref{p_phelix_ex}), a possible $T_s$ fixes the first two coordinates of $\alpha(t)+v\alpha'(t)$ and updates the third coordinate to the right-hand side of (\ref{p_phelix_ex}).
\end{proof}

\begin{remark}
    The tangent surface generated by the curve $\gamma(t) = (t, t^2, t^3)$ presented in Section \ref{sec_ahe} is one case where $T_s$ 
    satisfies (\ref{p_phelix_ex}) and $T_s(S) \subset S$. For the example in this section, it is not possible to ensure that $T_s(S) \subset S$ and $\mu_s, \nu$ are still supported on $S$, even if we only require $\overrightarrow{n}(s)$ to be any normal vector of $S$ at $\Phi(s, v)$ and have a norm comparable to 1. 
        
    One interpretation is that $T_s$ defined in the proof maps $S \to S_{s}$, where $S_s$ is the tangent surface generated by the curve
        $$\alpha_s(t) = \langle t, t^2+t^4, 6s^3(s-1)(t-s)+(1-6s^2)(t-s)^3-3s(t-s)^4\rangle,$$
        and
        $$T_s(\alpha(t)+v\alpha'(t)) = \alpha_s(t)+v\alpha_s'(t).$$
\end{remark}

    \begin{figure}[!h]
\centering
\beginpgfgraphicnamed{Illustration}
\begin{tikzpicture}
\draw[gray, thick] (0, 0) -- (3,0.5);
\draw[gray, thick] (1, 0.4) -- (3, 2.2);
\draw[gray, thick] (1.21, 1.5) -- (0.8, 3);
\filldraw[black] (2, 1.3) circle (2pt) node[anchor=west]{$\alpha(t)+v\alpha'(t)$};
\draw (0,0) .. controls (1, 0.2) and (2, 0.6) .. (0.5 ,3);
\node[draw=none] at (0.5,1) {$S$};

\draw[gray, thick] (10, 0) -- (13, 1.5);
\draw[gray, thick] (10.93, 0.9) -- (11.8, 3);
\draw[gray, thick] (10.75, 2) -- (9.9, 3);
\filldraw[black] (11.5, 2.3) circle (2pt) node[anchor=west]{$\alpha_s(t)+v\alpha_s'(t)$};
\draw (10,0) .. controls (11, 0.5) and (12, 2) .. (9 , 3);
\node[draw=none] at (10,1) {$S_{s}$};

\draw[->]  (5, 1) -- (9, 1)
node [above,text width=3cm,text centered,midway]
{$T_s$
};

\draw[black, fill=gray, fill opacity=0.3] (2, 1.3) -- (2.4, 1.7) -- (2.4,2) -- (2,1.6) -- cycle;

\draw[black, fill=black, fill opacity=0.6] (11.5, 2.3) -- (11.8, 3) -- (11.6, 3.2) -- (11.3, 2.45) -- cycle;
\end{tikzpicture}
\endpgfgraphicnamed
\caption{Illustration of the map $T_s:S \to S_{s}$.}
\end{figure}
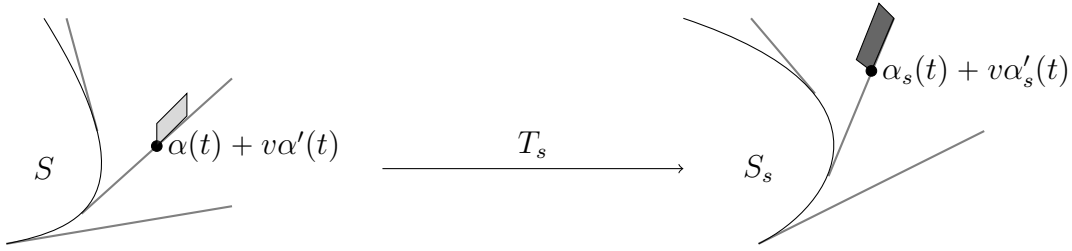

\section{Proof of Proposition \ref{p_fd_leq_k}}
\label{sec764}

We use terminologies commonly used in describing hypersurfaces of constant rank.

\begin{definition}
$M$ is a $(d-k)$\textit{-ruled} hypersurface if near $p \in M$, a coordinate chart is given by $\Phi: U \times V \subset \RR^{k} \times \RR^{d-k} \to S$ as 
    \begin{equation}
        \Phi(u, v) = \alpha(u) + \sum_{l=1}^{d-k} v_l w_l(u),
    \end{equation}
    with $\alpha, w_1, \ldots, w_{d-k}:\RR^{k} \to \RR^{d+1}$. A \textit{ruling} is $\{\Phi(u, v) | v\in V\}$ for a fixed $u$. 
\end{definition}

Let $N(u, v)$ be the normal vector of $M$ at $\Phi(u, v)$. In the proof of Proposition \ref{p_fd_leq_k}, we rely on the following parametrization on hypersurfaces with constant rank $k$. 

\begin{lem}
\cite[Lemma 3.1]{Hartman} \cite{Ushakov}
\label{lem_hartman}
The following statements are equivalent.
\begin{enumerate}
    \item $M$ is a smooth hypersurface in $\RR^{d+1}$ of constant rank $k$.
    \item $M$ is $(d-k)$-ruled, and the normal vectors along the rulings are constant.
\end{enumerate}
\end{lem}

We deduce the analytic properties of constant rank hypersurfaces that will be applied throughout this section. 

\begin{enumerate}[a.]
    \item The normal vectors along the rulings being constant is equivalent to 
    \begin{equation}
    \label{stable_ruling}
        N(u, v) = N(u, v')
    \end{equation}
    for $v \neq v'$. We may denote $N(u) = N(u, v)$ for a $v \in V$.
    Since $N(u, v)$ is the normal vector at $p = \Phi(u, v)$, it is perpendicular to all vectors in $T_p M$. Therefore,
    \begin{equation}
    \label{Phi_first_d}
        \begin{split}
            \left\langle \frac{\partial \Phi}{\partial u_j}(u, v), N(u, v)\right\rangle & = \left\langle \frac{\partial \alpha}{\partial u_j}(u) + \sum_{l=1}^{d-k} v_l \frac{\partial w_l}{\partial u_j}(u), N(u, v)\right\rangle = 0,\\
            \left\langle \frac{\partial \Phi}{\partial v_l}(u, v), N(u, v)\right\rangle & = \left\langle w_l(u), N(u, v)\right\rangle = 0.
        \end{split}
    \end{equation}
    From the first equation of (\ref{Phi_first_d}), for a fixed $u$, since $N(u, v) = N(u, v')$ when $v \neq v'$, for all $l$ and $j$,
    $$\left\langle  \frac{\partial w_l}{\partial u_j}(u), N(u, v)\right\rangle = 0.$$
    By differentiating the second equation in (\ref{Phi_first_d}) with respect to $u_j$, we obtain \begin{equation}
    \label{eq_n_stable}
        \left\langle  w_l(u) , \frac{\partial N}{\partial u_j} (u, v)  \right\rangle = 0.
    \end{equation}
    \item The second fundamental form of $M$ at $p$ has rank $k$. Since $$\frac{\partial^2 \Phi}{\partial v_{l} \partial v_{l'}}(u, v) = 0,$$ and 
    $$\left\langle \frac{\partial^2 \Phi}{\partial u_{j} \partial v_{l}}(u, v), N(u, v)\right\rangle = \left\langle \frac{\partial w_l}{\partial u_j}(u), N(u, v)\right\rangle=0,$$
    the top-left $k \times k$ submatrix of the second fundamental form $$\left(\frac{\partial^2 \Phi}{\partial u_{j} \partial u_{j'}}(u, v) \cdot N(u, v)\right)$$ has full rank $k$.
\end{enumerate}
    
By rotation and translation, we may assume that $U$ is a neighborhood of $0$, $N(0) = e_{d+1} \in \RR^{d+1}$, where $\{e_1, \cdots, e_{d+1}\}$ is the standard basis of $\RR^{d+1}$. 

\begin{remark}
\begin{enumerate}[a.]
    \item A $(d-k)$-ruled hypersurface does not have a constant rank of $k$ if the normal vectors are not constant along the rulings, and an $m$-ruled hypersurface can also be viewed as an $(m-1)$-ruled hypersurface.  For example, in $\RR^3$, the hyperboloid given by $x^2+y^2=z^2+1$ is $1$-ruled, but it is a hypersurface with non-vanishing Gaussian curvature. \cite[Sec. 3-5]{do_diff} The normal vectors of the hyperboloid are not constant along the rulings. 
    \item 
The proof of Theorem \ref{t33} uses Morse's Lemma (Lemma \ref{lem_morse}) to parametrize a hypersurface $M$. 
The author's previous work \cite{jz} partly uses Morse's Lemma to parametrize cones and cylinders. A version of Proposition \ref{p_fd_leq_k} could also be proven via Morse's Lemma; however, the conclusion would be the same as the Fourier dimension is independent of the parametrizations.
\item The hypersurface $M$ is cylindrical if $w_l(u_1) = w_l(u_2)$ for all $u_1, u_2 \in U$, $1 \leq l \leq d-k$. There is a simpler proof in Lemma \ref{lem_fd_cart} for a cylindrical hypersurface since it can be written as  $S \times \RR^{d-k}$ for a hypersurface $S \subset \RR^{k+1}$ with non-vanishing Gaussian curvature, and it follows that $\dim_F(M) = \dim_F(S) = k$.
\end{enumerate}

\end{remark}

\subsection{The Main Proof}

Proposition \ref{p_fd_leq_k} holds if the following lemma holds.

\begin{lemma}
\label{l_Ts_ex}
    For $s \in U$, there exists $T_s: M \to \RR^{d+1}$, such that
    \begin{equation}
    \label{e621}
     T_s(\Phi(u, v)) \cdot e_{d+1}  = (\Phi(u, v) - \alpha(s)) \cdot N(s)
    \end{equation}
    for the unit normal vector $N(s)$.
\end{lemma}

\begin{proof}[Proof of Proposition \ref{p_fd_leq_k} assuming Lemma \ref{l_Ts_ex}]

Let $\mu \in \mathcal{M}(M)$, and suppose that there exists  $\beta > 0$, such that $|\widehat{\mu}(\xi)| \leq C |\xi|^{-\frac{\beta}{2}}$ for $\xi \in \RR^{d+1}$, $C>0$. By Lemma \ref{lem_reduct}, we may assume $\mu \in \mathcal{M}(\Phi(\frac{1}{2}U, V))$. 
    
    We construct new measures. Let $\mu_s \in \mathcal{M}(\RR^{d+1})$ be 
    \begin{equation}
        \label{eq4_mus}
        \mu_s = T_s^{\#} \mu.
    \end{equation}
    Let $\psi \in C_{0}^{\infty}(\RR^{k})$ be non-negative bump function with $\psi(s)= 1$ when $s \in \frac{1}{2} U$, $\psi(s)=0$ when $s \in (\frac{3}{4}U)^c$. Then, we define the average measure $\nu \in \mathcal{M}(\RR^{d+1})$ as
    \begin{equation}
        \label{eq4_nu}
        \int f d\nu = \int \langle f, \mu_s \rangle \psi(s) ds
    \end{equation}
    for non-negative Borel functions $f$ via the Riesz representation theorem for positive linear functionals \cite{rudin1987real}. Proposition \ref{p_fd_leq_k} is a consequence of the following two lemmas. 

    \begin{lemma}
        \label{l_nuhat_beta}
        For $\rho > 0$, $|\widehat{\nu}(\rho e_{d+1})| \lesssim \rho^{-\frac{\beta}{2}}$.
    \end{lemma}

    \begin{lemma}
        \label{l_nuhat_k2}
        There exists $\rho_0 > 0$, such that for $\rho > \rho_0$,
        $|\widehat{\nu}(\rho e_{d+1})| \gtrsim \rho^{-\frac{k}{2}}$. \qedhere
    \end{lemma}
\end{proof}

\begin{proof}[Proof of Lemma \ref{l_nuhat_beta}]
    First, we show that for each $s\in U$, $|\widehat{\mu_s}(\rho e_{d+1})| \lesssim \rho^{-\frac{\beta}{2}}$. By unwinding the definitions of the push-forward measure, Fourier transform, and integration on the manifold, we have
    \begin{equation}
    \label{eq_mus_ex}
        \begin{aligned}
            \widehat{\mu_s}(\rho e_{d+1}) & = \int e^{-2\pi i y \cdot \rho e_{d+1}} dT_s^{\#}\mu(y) & \text{ by (\ref{eq4_mus})}\\
            & = \int e^{-2\pi i \rho T_s(\Phi(u, v)) \cdot  e_{d+1}} d\mu\circ\Phi(u, v) \\
            & = \int e^{-2\pi i \rho [ \Phi(u, v)  \cdot N(s) - \alpha(s) \cdot N(s)]} d\mu\circ \Phi(u, v) & \text{ by (\ref{e621})} \\
            & = e^{2 \pi i \rho \alpha(s) \cdot N(s)} \widehat{\mu}(\rho N(s)).
          \end{aligned}
    \end{equation}
    Since $|N(s)| = 1$, $|\widehat{\mu_s}(\rho e_{d+1})| = |\widehat{\mu}(\rho N(s))| \leq  C\rho^{-\frac{\beta}{2}}.$ Then
        \begin{align*}
            |\widehat{\nu}(\rho e_{d+1})| \leq & \int |\widehat{\mu_s}(\rho N(s))| \psi(s) ds & \text{ by (\ref{eq4_nu})} \\
            \leq & C \norm{\psi}_{L^{1}} \rho^{-\frac{\beta}{2}}. && \qedhere
        \end{align*} 
\end{proof}

\begin{proof}[Proof of Lemma \ref{l_nuhat_k2}]
An expression of $\widehat{\nu}(\rho e_{d+1})$ is given by
\begin{equation}
\label{eq_nu_ex}
    \begin{aligned}
        \widehat{\nu}(\rho e_{d+1}) & = \int \int e^{-2\pi i \rho T_s(\Phi(u, v)) \cdot  e_{d+1}} d\mu\circ\Phi(u, v) \psi(s) ds & \text{ by (\ref{eq4_nu})}\\
        & = \int \int e^{-2\pi i \rho T_s(\Phi(u, v)) \cdot  e_{d+1}}  \psi(s) ds d\mu \circ \Phi(u, v).
    \end{aligned}
\end{equation}

We apply the stationary phase method to study the inner integral
\begin{equation*}
   I(\rho; u, v) := \int e^{-2\pi i \rho T_s(\Phi(u, v)) \cdot  e_{d+1}}  \psi(s) ds.
\end{equation*}

Fixing $u, v$, we define the phase function $\phi$ of $I(\rho; t, v)$ as
$$\phi(s) := T_s(\Phi(u, v)) \cdot  e_{d+1} = \Phi(u, v)  \cdot N(s) - \alpha(s) \cdot N(s) \text{ by }(\ref{e621}).$$ 
Two claims about the critical point of $\phi$ are as follows:
\begin{itemize}
    \item One critical point is $s=u$: by (\ref{eq_n_stable}),
    \begin{equation*}
        \begin{split}
            \frac{\partial \phi}{\partial s_j}(u) = & (\Phi(u, v) - \alpha(u)) \cdot \frac{\partial N}{\partial u_j}(u) \\
            = &  \left(\alpha(u) + \sum_{l=1}^{d-k} v_l w_l(u) - \alpha(u) \right) \cdot \frac{\partial N}{\partial u_j}(u) \\
            = & \sum_{l=1}^{d-k} v_l w_l(u) \cdot \frac{\partial N}{\partial u_j}(u) = 0.
        \end{split}
    \end{equation*}
    \item If $U$ is sufficiently small, no other critical point exists: we will show that the Hessian has full rank, so the critical points of the function $\phi$ are isolated. By differentiating (\ref{eq_n_stable}) with respect to $u_{j'}$, we obtain
    $$  \left\langle  \frac{\partial w_l}{\partial u_{j'}}(u) , \frac{\partial N}{\partial u_j} (u)  \right\rangle +  \left\langle  w_l(u) , \frac{\partial^2 N}{\partial u_j\partial u_j'} (u)  \right\rangle = 0$$
   Then,
  \begin{equation*}
    \begin{split}
     \frac{\partial^2 \phi}{\partial s_j \partial s_{j'}} (u) = & (\Phi(u, v)-\alpha(u) )\cdot \frac{\partial^2 N}{\partial u_j \partial u_{j'}}(u) - \frac{\partial \alpha}{\partial u_{j'}}(u) \cdot \frac{\partial N}{\partial u_j}(u) \\
      = & - \frac{\partial \Phi}{\partial u_j'}(u, v) \cdot \frac{\partial N}{\partial u_j}(u).
    \end{split}
\end{equation*}
By part b) of the discussion following Lemma \ref{lem_hartman}, 
the matrix $\left(\frac{\partial^2 \Phi}{\partial u_{j} \partial u_{j'}}(u, v) \cdot N(u)\right)$ has full rank $k$. 
Therefore, $\Delta_s \phi(u)$ has full rank $k$.
\end{itemize}

Then, by the stationary phase (part b of Theorem \ref{t_asym}),
$$\left|I(\rho; u, v) -\rho^{-\frac{k}{2}}(-1)^{k}i^{\frac{k}{2}}\left(\det \left(\frac{\partial^2 \Phi}{\partial u_{j} \partial u_{j'}}(u, v) \cdot N(u)\right)\right)^{-\frac{1}{2}}\psi(u)\right| \leq D' \rho^{-\frac{k+1}{2}}$$
for a $D'>0$ independent of $\rho$, $u$, and $v$. Since $\psi(u) = 1$ for $ u \in \frac{1}{2}U$, continuing from (\ref{eq_nu_ex}),
\begin{equation*}
    \begin{split}
        & \left|\widehat{\nu}(\rho e_{d+1}) - \int \rho^{-\frac{k}{2}}(-1)^{k}i^{\frac{k}{2}}\left(\det \left(\frac{\partial^2 \Phi}{\partial u_{j} \partial u_{j'}}(u, v) \cdot N(u)\right)\right)^{-\frac{1}{2}} d\mu \circ \Phi(u, v) \right|\\
         \leq & \int D' \rho^{-\frac{k+1}{2}}  d\mu \circ \Phi(u, v).\\
    \end{split}
\end{equation*}

Since $$\det \left(\frac{\partial^2 \Phi}{\partial u_{j} \partial u_{j'}}(u, v) \cdot N(u) \right) \leq D$$ for a $D> 0$,
\begin{align*}
    \left|\int \rho^{-\frac{k}{2}}(-1)^{k}i^{\frac{k}{2}}\left(\det \left(\frac{\partial^2 \Phi}{\partial u_{j} \partial u_{j'}}(u, v) \cdot N(u)\right)\right)^{-\frac{1}{2}} d\mu \circ \Phi(u, v) \right| \geq & D^{-\frac{1}{2}}\rho^{-\frac{k}{2}}, \\
    \left|\int D' \rho^{-\frac{k+1}{2}}  d\mu \circ \Phi(u, v)\right| \leq &  D' \rho^{-\frac{k+1}{2}}.
\end{align*}
Therefore, $|\widehat{\nu}(\rho e_{d+1})|$ is bounded below by $2^{-1}D^{-\frac{1}{2}} \rho^{-\frac{k}{2}}$ for $\rho$ sufficiently large.
\end{proof}

\subsection{Existence of $T_s$}
\label{sec_ex}

\begin{proof}[Proof of Lemma \ref{l_Ts_ex}]
    One choice of $T_s$ fixes the first $n$ coordinates of $\Phi(u, v)$ and changes the last coordinate to $$(\Phi(u, v) - \alpha(s)) \cdot N(s) = (\alpha(u)- \alpha(s))\cdot N(s) + \sum_{l=1}^{d-k} v_l w_l(u) \cdot N(s).$$
    Alternatively,
    \begin{equation}
    \label{Ts_def}
    \begin{split}
         T_s(\Phi(u, v)) = \alpha_s(u) + \sum_{l=1}^{d-k} v_l w_{l, s}(u),
    \end{split}
    \end{equation}
    where 
    \begin{align*}
         \alpha_s(u) = & \alpha(u) + e_{d+1}[\alpha(u) \cdot (N(s) -e_{d+1}) - \alpha(s) \cdot N(s)], \\
         w_{l, s}(u) = & w_{l}(u) + e_{d+1} [w_{l}(u) \cdot (N(s) - e_{d+1})]. & & \qedhere
    \end{align*}
\end{proof}

\begin{remark}
     The new hypersurface $M_s := T_s(M)$ is also $d-k$ ruled with a coordinate chart $\Phi_s := T_s \circ \Phi$ if $s$ is sufficiently small. Since for all $l$ and $j$, $p = T_s(\Phi(u, v))$,
       $$\frac{\partial w_{l, s}}{\partial u_j}(u) \in T_{p}M_s,$$ 
       for $T_s$ defined in (\ref{Ts_def}), the normal vectors of $T_s(M)$ are stable along the rulings, so $M_s$ is also a smooth hypersurface of constant rank $k$.
\end{remark}

\section{Appendix}
\label{sec_app}

\subsection{Hausdorff and Fourier dimensions}
\label{hf_dim}

One notion of size often used in fractal geometry is the \textit{Hausdorff dimension}, which is defined as follows. For a set $A \subset \RR^n$, $s, \delta > 0$, 
$$\mathcal{H}^{s}_{\delta}(A) := \inf \left\{\left.\sum_{j} \text{diam}(E_j)^s \right| A \subset \bigcup_{j} E_j, \text{diam}(E_j) < \delta\right\},$$
and the \textit{$s$-dimensional Hausdorff measure} $\mathcal{H}^{s}(A) := \lim_{\delta \to 0} \mathcal{H}^{s}_{\delta}(A)$. The \textit{Hausdorff dimension} of $A$ is $\dim_H(A) := \sup \{s: \mathcal{H}^{s}(A) = \infty\} = \inf \{s: \mathcal{H}^{s}(A) = 0 \}.$

Frostman's lemma \cite[Theorem 2.7]{mattila_2015} offers an alternative characterization of the Hausdorff dimension.  For a Borel set $A \subset \RR^{n}$, 
$$\dim_H(A) = \sup \{s \in [0, n]: \exists \mu \in \mathcal{M}(A), I_s(\mu) := \int |\widehat{\mu}(\xi)|^2 |\xi|^{s-n} d\xi < \infty \},$$ where $I_s$ is the $s$\textit{-energy} of $\mu$.

For $\mu \in \mathcal{M}(A)$, if $\sup_{\xi \in \RR^{n}}|\xi|^{\frac{s}{2}}|\widehat{\mu}(\xi)| < \infty$ for an $s > 0$, $I_{s_0}(\mu) < \infty$ when $0<s_0 < s$. From the characterization above, $\dim_H(A) \geq s$.
Consequently, $\dim_F(A) \leq \dim_H(A)$, and the inequality is strict for some sets $A$. The set $A$ is \textit{Salem} if $\dim_F(A)=\dim_H(A)$. 



\subsection{Fourier dimension of the Cartesian product} We present a lemma that yields the Fourier dimension of cylindrical hypersurfaces.

\begin{lemma}
\label{lem_fd_cart}
Suppose that $S \subset \RR^n$, $T \subset \RR^m$ are compact, $\dim_F(S) < n$, and $\dim_F(T) < m$. Then,
\begin{enumerate}[a.]
    \item  
    \begin{equation}
    \label{eq48}
      \dim_F(S \times T) = \min\{\dim_F(S), \dim_F(T)  \},  
    \end{equation}
    \item 
    \begin{equation}
      \dim_F(S \times \RR^{m}) = \dim_F(S).  
    \end{equation}
\end{enumerate}
\end{lemma}

We refer readers to the discussion in \cite{falconer} on the Hausdorff dimension of the Cartesian product.

\begin{proof}
    For part a, let $s := \dim_F(S) < n$, $t := \dim_F(T) < m$. We write $x = (x_1, x_2) \in \RR^{n} \times \RR^{m}$.
    
    \begin{itemize}
        \item First, we show that $$\dim_F(S \times T) \geq \min\{s, t \}.$$

    Let $\eps > 0$. Let $\mu_S \in \mathcal{M}(S)$, with $|\widehat{\mu_S}(\xi_1)| \leq C_1 |\xi_1|^{-\frac{s}{2}+\eps }$ for $C_1 > 0$, $\xi_1 \in \RR^{n}$, $|\xi_1| \geq 1$. 
    
    Similarly, let $\mu_T \in \mathcal{M}(T)$, with $|\widehat{\mu_T}(\xi_2)| \leq C_2 |\xi_2|^{-\frac{t}{2}+\eps}$ for $C_2 > 0$, $\xi_2 \in \RR^m$, $|\xi_2| \geq 1$. We consider the measure $\mu_S \times \mu_T \in \mathcal{M}(S \times T)$.  For $\xi = (\xi_1, \xi_2) \in \RR^{n} \times \RR^{m}$,
    \begin{equation*}
        \begin{split}
            \widehat{\mu_S \times \mu_T} (\xi) & = \int_{\RR^{n+m}} e^{- 2 \pi i x \cdot \xi} d(\mu_S \times \mu_T) (x) \\
            = & \left(\int_{\RR^{n}} e^{-2\pi i x_1 \cdot \xi_1} d\mu_S(x_1) \right) \left(\int_{\RR^{m}} e^{-2\pi i x_2 \cdot \xi_2} d\mu_T(x_2) \right) \\
            = & \widehat{\mu_S}(\xi_1)\widehat{\mu_T}(\xi_2).
        \end{split}
    \end{equation*}
    There are two cases to consider.

    \begin{itemize}
        \item Case 1: $|\xi_1| \geq \frac{1}{2}|\xi|$. Then, $$|\widehat{\mu_S \times \mu_T} (\xi)| = |\widehat{\mu_S}(\xi_1)\widehat{\mu_T}(\xi_2)| \leq C_1 |\xi_1|^{-\frac{s}{2}+\eps} \leq C_1 2^{\frac{s}{2}-\eps} |\xi|^{-\frac{s}{2}+\eps}.$$
        \item  Case 2: $|\xi_2| \geq \frac{1}{2}|\xi|$.
    Similarly, $|\widehat{\mu_S \times \mu_T} (\xi)| \leq C_2 2^{\frac{t}{2}-\eps} |\xi|^{-\frac{t}{2}+\eps}.$
    \end{itemize}
    
    This shows that $\dim_F(S \times T) \geq \min\{s-2\eps, t-2\eps  \}$, then we can let $\eps \to 0$.

    \item Next, we show that 
    $$\dim_F(S \times T) \leq \min\{s, t\}.$$

    Let $\mu \in \mathcal{M}(S \times T)$, and suppose that $|\widehat{\mu}(\xi)| \leq C |\xi|^{-\frac{r}{2}}$ for $\xi = (\xi_1, \xi_2) \in \RR^{n} \times \RR^{m}$. We will show that $r \leq s$ and $r \leq t$. 
    \begin{itemize}
        \item Let $\pi_S : S \times T \to S$ be the projection map to the first $n$ coordinates and $\pi_T : S \times T \to T$ be the projection map to the last $m$ coordinates. Since $S$ and $T$ are compact, for $\mu \in \mathcal{M}(S \times T)$, $$\text{spt} \pi_S^{\#} \mu = \pi_S(\text{spt} \mu) \subset S,$$ $$\text{spt} \pi_T^{\#} \mu = \pi_T(\text{spt} \mu) \subset T.$$
        

        
        \item First, consider $\xi = (\xi_1, 0)$ (so $\xi_2 = 0$). Then,
        \begin{equation}
        \label{eq84}
            \begin{split}
                \widehat{\mu}(\xi) & = \int e^{-2 \pi i x \cdot \xi} d\mu(x) \\
                & = \int e^{-2 \pi i x_1 \cdot \xi_1} d\mu(x) \\
                & = \int e^{-2 \pi i \pi_S(x) \cdot \xi_1} d\mu(x) \\
                & = \int e^{-2 \pi i y \cdot \xi_1} d\pi_S^{\#}\mu(y) \\
                & = \widehat{\pi_S^{\#}\mu} (\xi_1)
            \end{split}
        \end{equation}
        Since $\pi_S^{\#}\mu \in \mathcal{M}(S)$, $\dim_F(S) = s$, $|\widehat{\pi_S^{\#}\mu} (\xi_1)| = |\widehat{\mu}(\xi)| \leq C |\xi|^{-\frac{r}{2}} = C |\xi_1|^{-\frac{r}{2}}$, then $r \leq s$.
        \item Next, consider $\xi = (0, \xi_2)$ (so $\xi_1 = 0$). Similar to (\ref{eq84}), $|\widehat{\mu}(\xi)| = |\widehat{\pi_T^{\#}\mu} (\xi_2)|$.  As $\pi_T^{\#}\mu \in \mathcal{M}(T)$, $\dim_F(T) = t$, $|\widehat{\pi_T^{\#}\mu} (\xi_2)| = |\widehat{\mu}(\xi)| \leq C |\xi|^{-\frac{r}{2}} = C |\xi_2|^{-\frac{r}{2}}$, we have $r \leq t$.
    \end{itemize}
     \end{itemize}
     The proof for part b) is similar, where we can take $\mu_T \in \mathcal{S}(\RR^{m})$.
\end{proof}

\subsection{Reduction to compactly supported measure}

We present a lemma that allows us to assume that the measure $\mu \in \mathcal{M}(M)$ we study has smaller support in a smaller closed set on the manifold $M$.

\begin{lemma}\cite[Theorem 1]{eps}
\label{lem_reduct}
     Suppose that $\mu_0 \in \mathcal{M}(\RR^n)$, $\sup_{\xi \in \RR^{n}}|\xi|^{\alpha}|\widehat{\mu_0}(\xi)| < \infty$ for $\alpha >0$. Let $f \in \mathcal{S}(\RR^{n})$ with $f \geq 0$, and $\mu \in \mathcal{M}(\RR^n)$ such that $d\mu = fd\mu_0$. Then, $|\widehat{\mu}(\xi)| \leq c_{\mu} |\xi|^{-\alpha}$ for a $c_{\mu}>0$.
\end{lemma}

\begin{proof}
    Note that
    \begin{equation*}
    \begin{split}
        |\widehat{\mu}(\xi)| & = |\widehat{\mu_0} * \widehat{f}(\xi)| \\
        & = \left|\int \widehat{\mu_0}(\xi - \eta) \widehat{f}(\eta) d\eta  \right| \\
        & \leq \left|\int_{|\eta| \leq \frac{|\xi|}{2}} \widehat{\mu_0}(\xi - \eta) \widehat{f}(\eta) d\eta  \right| + \left|\int_{|\eta| \geq \frac{|\xi|}{2}} \widehat{\mu_0}(\xi - \eta) \widehat{f}(\eta) d\eta  \right|.
    \end{split}
    \end{equation*}

    For the first integral, since $|\eta| \leq \frac{|\xi|}{2}$, $|\xi-\eta| \geq \frac{|\xi|}{2}$, and $\widehat{f}$ is integrable, the integral is bounded above by a constant multiple of $|\xi|^{-\alpha}$. For the second integral, we apply the bounds $\norm{\widehat{\mu_0}}_{L^{\infty}}=1$ and $\int_{|\eta| \geq \frac{|\xi|}{2}} |\widehat{f}(\eta)| d\eta \leq c_m |\xi|^{-m}$ for $m \in \N$ with a $c_m > 0$ since $\widehat{f} \in \mathcal{S}(\RR^n)$. If we choose $m > \alpha$, the sum of two bounds is bounded by a constant multiple of $|\xi|^{-\alpha}$ for large $|\xi|$.
\end{proof}

\subsection{Morse's lemma}
\label{sec62}
The version stated below generalizes the one shown in \cite[Section 8.2]{Stein1993}. 

\begin{lemma}
\label{lem_morse}
Suppose that $f \in C^{\infty}(\RR^d \times \RR^{d})$, $f(0, t)=0$, $\nabla_x f(0, t) = 0$, and that $\Delta_xf(x, t)$ has a constant rank of $n$. Then, there exist neighborhoods $V, W$ of $0$ and a smooth $\tau: V \times W \to \RR^{d}$ such that $\tau(0, t)=0$, $\det {\bf J}_x \tau(x, t) \neq 0$, and
\begin{equation}
\label{eq_f_qmt}
    f(x, t) = Q_{m, n}(\tau(x, t)),
\end{equation}
where 
\begin{equation}
\label{eq_Qm}
  Q_{m, n}(y) : =\sum_{j=1}^{m} y_j^2 - \sum_{j=m+1}^{n} y_j^2  
\end{equation}  with $0 \leq m \leq n \leq d$.
\end{lemma}

\begin{remark}
    The number of positive eigenvalues of the matrix $\Delta_xf(0, 0)$ is $m$. By differentiating (\ref{eq_f_qmt}) twice,
\begin{equation}
    \label{eq_JDJD}
     [{\bf J}_x\tau(0, t)]^T\Delta Q_{m, n}(0)  {\bf J}_x\tau(0, t) = \Delta_x f(0, t).
\end{equation}
\end{remark}

\begin{proof}
We claim that the function $\tau$ can be expressed as
\begin{equation}
\label{mor_tau_ld1}
    \tau = L \circ \tau_d \circ \cdots \circ \tau_{1},
\end{equation}
where each $\tau_r$ is a change of variables in the first $d$ coordinates and where $L$ is a permutation of the same coordinates. $\tau_r$ is constructed inductively as follows: suppose that at step $r$, we have
$$\widetilde{g_r}(x^{}, t) := f(\tau_{1}^{-1}\circ \cdots \circ \tau_{r-1}^{-1}(x^{}, t)) = \pm [x^{}_1]^2 \pm \cdots \pm [x^{}_{r-1}]^2+ \sum_{j, k \geq r}^{d} x^{}_j x^{}_k \widetilde{f^{}_{j, k}}(x^{}, t),$$
where $$\widetilde{f^{}_{j, k}}(x^{}, t)= \int_{0}^{1}(1-s) \frac{\partial^2 \widetilde{g_r}}{\partial x^{}_j\partial x^{}_k}(sx^{}, t)ds, \widetilde{f^{}_{j, k}}(0, t) = \frac{1}{2} \frac{\partial^2  \widetilde{g_r}}{\partial x^{}_j \partial x^{}_k}(0, t).$$ 

When $r > n$, since $\Delta_x \widetilde{g_r} (x, t)$ only has rank $n$, all $\widetilde{f^{}_{j, k}}(x, t)=0$, and we are done. Otherwise, there exists
 an orthonormal matrix $O_r$, which is a linear change in the variables $x^{}_r, \cdots, x^{}_d$, such that 
$$g_r(y^{},t) :=  f(\tau_{1}^{-1}\circ \cdots \circ \tau_{r-1}^{-1}(O_r y^{}, t)) = \pm [y^{}_1]^2 \pm \cdots \pm [y^{}_{r-1}]^2+ \sum_{j, k \geq r}^{d} y^{}_j y^{}_k f^{}_{j, k}(y^{}, t),$$
where $$f^{}_{j, k}(y^{}, t)= \int_{0}^{1}(1-s) \frac{\partial^2 g_r}{\partial y^{}_j\partial y^{}_k}(sy^{}, t)ds, f^{}_{j, k}(0, t) = \frac{1}{2} \frac{\partial^2  g_r}{\partial y^{}_j \partial y^{}_k}(0, t)$$ and the additional condition that $f^{}_{r, r}(0, t)  \neq 0$. Then, for each $t$, we can perform a change of variables from $y^{}$ to $z^{}$ such that $z^{}_j = y^{}_j$ for $j \neq r$, and
\begin{equation}
\label{mor_var_ch}
    z^{}_r = [\pm f^{}_{r,r}(y^{}, t)]^{\frac{1}{2}}\left[y^{}_r + \sum_{j > r} \frac{y^{}_j
f^{}_{j, r}(y^{}, t)}{\pm f^{}_{r, r}(y^{}, t) } \right],
\end{equation}
where $\pm$ is the sign of $f^{}_{r, r}(0, t)$. This change of variables can be expressed as $z^{}=\sigma_{r}(y^{}, t)$ for $\sigma_r: V_r \times W_r \to \RR^{d}$, where $V_r, W_r$ are neighborhoods of $0$, $ f^{}_{r,r}$ does not change sign on $V_r \times W_r$, and $$\det {\bf J}_{y^{}} \sigma_{r}(y^{}, t) \neq 0.$$ Then, we let 
\begin{equation}
    \label{mor_taur}
    \tau_r(x^{}, t)= (\sigma_{r}(O_r^{-1} x^{}, t), t),
\end{equation}
and proceed to the next step by noting that $\tau_r^{-1}$ exists in a neighborhood of $0$.

From the construction above, there exist $V, W$  neighborhoods of $0$, such that for $(x, t) = (x^{}, t) \in V \times W$, for all $r$ from $1$ to $d$,  $f^{}_{r,r}(y^{}, t)$ does not change sign, and $$\det {\bf J}_{y^{}} \sigma_{r}(y^{}, t) \neq 0.$$ Therefore, $\tau$ is defined on $V \times W$. 
\end{proof}

Let $\tau_t(x) = \tau(x, t)$. For $k \geq 0$, $t \in W$, it is possible to bound $\norm{{\bf J}\tau_t}_{C^{k}(V)}$ and $\inf\{|\det {\bf J  }\tau_t(x)| x \in V \}$ via $\norm{f_t}_{C^{k+2}}$, $|\det {\bf J}\tau_t(0)|$, and the size of $V$ via (\ref{mor_tau_ld1}), (\ref{mor_var_ch}), and (\ref{mor_taur}).

\subsection{Oscillatory Integrals}

We refer readers to the complete proof of oscillatory integrals results in \cite{Stein1993}. In this section, we outline the key steps and bounds of the error terms.

For $f \in C^{\infty}(\RR^{n})$, $U$ open in $\RR^{n}$, and $k\geq 0$, we denote $\norm{f}_{C^k(U)}$, or $\norm{f}_{C^k}$ if the implication of the open set $U$ is clear, as the quantity
$$\sum_{|\beta| \leq k} \norm{\frac{\partial^{|\beta|}}{\partial y^{\beta}} f}_{L^{\infty}(U)}.$$

\begin{proposition}\cite[Chapter 8, Propositions 1 and 4]{Stein1993}
\label{p_local}
(Localization)
    \begin{enumerate}[a.]
        \item (Single variable edition) For $N \in \N$, there exist $C_N >0$, $p_N, q_N \in \N$ such that, for $\phi \in C^{\infty}(\RR)$, $\psi \in C^{\infty}_{0}(\RR)$ with $|\phi'(x)| \geq  1$ in $\text{spt }\psi$, $\lambda \geq 1$,
        $$\left|\int e^{i \lambda \phi(x)} \psi(x) dx\right| \leq   C_N\lambda^{-N} |\text{spt }\psi|\norm{\psi}_{C^{N}}^{p_N}\norm{\phi''}_{C^{N-1}}^{q_N}$$
        \item (Multivariate edition) For $N \in \N$, there exist $q_N, r_N\in \N$, such that for $\phi \in C^{\infty}(\RR^n)$, $\psi \in C^{\infty}_{0}(\RR^n)$ with $|\nabla \phi(x)| \geq  1$ in $\text{spt }\psi$,  $\Lambda$ being the maximum absolute values of the eigenvalues of $\Delta \phi(x')$ for all $x' \in \text{spt } \psi$, $\lambda \geq 1$, there exists $C_{N, \psi}>0$ with
        $$\left|\int e^{i \lambda \phi(x)} \psi(x) dx\right| \leq   C_{N, \psi}\lambda^{-N} (1+\Lambda)^{r_N} 
         \max_{\xi \in \mathbb{S}^{n-1}} \norm{\frac{\partial^2 \phi}{\partial \xi^2}}_{C^{N-1}}^{q_N} .$$
    \end{enumerate}
\end{proposition}

\begin{proof}
    \begin{enumerate}[a.]
        \item  Using integration by parts $N$ times, the integral can be written as
        \begin{equation*}
             \int e^{i \lambda \phi(x)} (^tD)^N\{ \psi(x) \}dx
        \end{equation*}
        with $^tD g = i \lambda^{-1} \frac{d}{d x}\left(\frac{g}{\phi'}\right)$. Note that
        \begin{equation*}
           \lambda^{N} (^tD)^N g = \frac{g^{(N)}}{(\phi')^N} + \sum_{r=1}^{N}  (\phi')^{-(N+r)} F_{N, r}(g, g', \cdots, g^{(N-1)}, \phi'', \cdots, \phi^{(N+1)})
        \end{equation*} for some polynomials $F_{N, r}$. Therefore, the integral is bounded above by
        \begin{equation*}
            C_N\lambda^{-N} |\text{spt }\psi|\norm{\psi}_{C^{N}}^{p_N}\norm{\phi''}_{C^{N-1}}^{q_N}
        \end{equation*}
        for $C_N > 0$, $p_N, q_N \in \N$.
        \item For each $x_0$ in the support of $\psi$, there is a unit vector $\xi_0 = (\nabla \phi)(x_0) / \norm{(\nabla \phi)(x_0)}$ and a ball centered $B(x_0)$ at $x_0$ with
        \begin{equation}
        \label{ball_cond}
            \xi_0 \cdot (\nabla \phi)(x) \geq 2^{-1}
        \end{equation}
        if $x \in B(x_0)$. We claim that $B(x_0)$ can be chosen with the radius bounded below by a uniform constant depending on $\phi$. By the mean-value theorem, there exists $x'$ on the line segment joining $x_0$ and $x$, such that 
        $$| \xi_0 \cdot (\nabla \phi)(x) -  \xi_0 \cdot (\nabla \phi)(x_0)| = |(x-x_0) \Delta \phi(x') \xi_0|.$$
        If $$|x-x_0| \leq (2(1+\Lambda))^{-1},$$
        then (\ref{ball_cond}) holds.
        
        We choose a finite cover $\{B_k\} \subset \{B(x_0)\}$, and write 
        $\psi = \sum_{k} \psi_k$ for $\text{spt }\psi_k \subset B_k$. For each $B_k$, we choose a coordinate system where $x_1$ is parallel to $\xi_k$. Then,
        $$\int e^{i \lambda \phi(x)} \psi_k(x) dx = \int \left(  \int e^{i \lambda \phi(x_1, \cdots, x_n)} \psi_{k}(x_1, \cdots, x_n) dx_1\right)dx_2 \cdots dx_n,$$
        and we can apply part a) to the inner integral and bound it above by
        $$ C_N\lambda^{-N} |\text{spt }\psi_k|\norm{\psi_k}_{C^{N}}^{p_N}\norm{\frac{\partial^2 \phi}{\partial x_1^2}}_{C^{N-1}}^{q_N}.$$
        When all $k$'s are summed, a bound on the original integral is
         $$ C_N'\lambda^{-N} |\text{spt }\psi|\norm{\psi}_{C^{N}}^{p_N} (1+\Lambda)^{r_N} 
         \max_{\xi \in \mathbb{S}^{n-1}} \norm{\frac{\partial^2 \phi}{\partial \xi^2}}_{C^{N-1}}^{q_N} $$
         for $C_N' >0$, and $r_N \in \N$. \qedhere
    \end{enumerate}
\end{proof}

\begin{lemma}\cite[Chapter 8, Proposition 6]{Stein1993} Let $Q_{m, n}(y)$ be defined in (\ref{eq_Qm}).
    \begin{enumerate}[a.]
    \item  If $\eta \in C^{\infty}_{0}(\RR^{n})$,
        \begin{equation}
        \label{eq_int_Qmyl}
            \left|\int_{\RR^{n}} e^{i \lambda Q_{m, n}(y)} y^{l} \eta(y) dy \right| \leq A_l \lambda^{-\frac{n+|l|}{2}},
        \end{equation}
        where $A_l>0$, $l \in \ZZ^{n}$, $l_j \geq 0$, and $|l| = \sum_{j=1}^{n} l_j$.
    \item  If $g \in \mathcal{S}(\RR^{n})$ and there exists $\delta > 0$, such that $g(y)=0$ for $y \in B(0, \delta)$, then for $N \in \N$, there exists $B_N > 0$ such that
        \begin{equation}
        \label{eq_int_Qmg}
            \left|\int_{\RR^{n}} e^{i \lambda Q_{m, n}(y)} g(y) dy \right| \leq B_N \lambda^{-N}.
        \end{equation}
    \end{enumerate}
\end{lemma}

\begin{proof}
\begin{enumerate}[a.]
    \item Consider the cones 
    $$\Gamma_k = \left\{y \in \RR^{n}| |y_k|^2 \geq \frac{|y|^2}{2n}  \right\},$$
    and  $$\Gamma_k^{0} = \left\{y \in \RR^{n}| |y_k|^2 \geq \frac{|y|^2}{n}  \right\}.$$
    Since $\cup_{k=1}^{n} \Gamma_k^{0} = \RR^{n}$, there are functions $\{\Omega_k\}_{1 \leq k \leq n}$ such that each $\Omega_k$ is homogeneous of degree $0$, smooth away from the origin, $0 \leq \Omega_k \leq 1$ with $$\sum_{k=1}^{n} \Omega_k(x) = 1$$ for $x\neq 0$, and each $\Omega_k$ is supported in $\Gamma_k$. Then,
    \begin{equation*}
        \int_{\RR^{n}} e^{i \lambda Q_{m, n}(y)} y^{l} \eta(y) dy = \sum_{k=1}^{n} \int_{\Gamma_{k}} e^{i \lambda Q_{m, n}(y)} y^{l} \eta(y) \Omega_k(y) dy.
    \end{equation*}
    
     In the cone $\Gamma_k$, we will show that there exists $A_{l, k}>0$ such that 
     \begin{equation}
     \label{eq_int_QmylOmegak}
       \left| \int_{\Gamma_{k}} e^{i \lambda Q_{m, n}(y)} y^{l} \eta(y) \Omega_k(y) dy\right| \leq A_{l, k} \lambda^{-\frac{n+|l|}{2}}.  
     \end{equation}
     Summing over all $k$ yields (\ref{eq_int_Qmyl}). Let $\alpha \in C^{\infty}(\RR^n)$, such that $\alpha(y) = 1$ for $|y| \leq 1$, and $\alpha(y)=0$ for $|y| \geq 2$. Then, for $\epsilon>0$,
        \begin{equation*}
            \begin{split}
            & \int_{\Gamma_{k}} e^{i \lambda Q_{m, n}(y)} y^{l} \eta(y) \Omega_k(y) dy  \\
            = & \int_{\Gamma_{k}} e^{i \lambda Q_{m, n}(y)} y^{l} \eta(y) \Omega_k(y) \alpha(\epsilon^{-1} y) dy + \int_{\Gamma_{k}} e^{i \lambda Q_{m, n}(y)} y^{l} \eta(y) \Omega_k(y) [1-\alpha(\epsilon^{-1} y)] dy.
            \end{split}    
        \end{equation*}
        For the first integral,
        \begin{equation*}
            \left|\int_{\Gamma_{k}} e^{i \lambda Q_{m, n}(y)} y^{l} \eta(y) \Omega_k(y) \alpha(\epsilon^{-1} y) dy  \right| \lesssim \norm{\eta}_{L^{\infty}} \epsilon^{|l|+1}.
        \end{equation*}
        Let $N \in \N$. Using integration by parts $N$ times, the second integral can be written as
        \begin{equation}
        \label{eq_int_tDN}
             \int_{\Gamma_k} e^{i \lambda Q_{m, n}(y)} (^tD_k)^N\{ y^l \eta(y)  \Omega_k(y)  [1-\alpha(\epsilon^{-1} y)] \}dy
        \end{equation}
        with $^tD_k g = s_k(2 i\lambda)^{-1} \frac{\partial}{\partial y_k}\left(\frac{g}{y_k}\right)$ for a differentiable function $g$, and $s_k=-1$ if $k \leq m$, $s_k=1$ if $k \geq m+1$. Note that
        \begin{equation}
        \label{eq_tDkNg}
            (^tD_k)^N g = \lambda^{-N}\sum_{r=0}^{N}a^{(m, k)}_{N, r} y_k^{r-2N}\frac{\partial^r g}{\partial y_k^r}
        \end{equation} for $a^{(m, k)}_{N, r} \in \CC$. When we apply (\ref{eq_tDkNg}) and the product rule of the derivative to expand (\ref{eq_int_tDN}), we obtain a summation of terms where a term, ignoring the constant, is
        \begin{equation*}
                \lambda^{-N}\int_{\Gamma_k \cap B(0, \epsilon)^{c}} e^{i \lambda Q_{m, n}(y)} y_k^{r-2N} \left[\frac{\partial^{r_1}}{\partial y_k^{r_1}} y^{l} \right]  \left[\frac{\partial^{r_2}}{\partial y_k^{r_2}}\eta(y)\right]\left[\frac{\partial^{r_3}}{\partial y_k^{r_3}}\Omega_k(y)\right]\left[\frac{\partial^{r_4}}{\partial y_k^{r_4}} [1-\alpha(\epsilon^{-1} y)] \right]dy
        \end{equation*}
        for $r_1, r_2, r_3, r_4 \geq 0$, $r_1+r_2+r_3+r_4 = r \leq N$. We note that $\frac{\partial^{r_3} \Omega_k}{\partial y_k^{r_3}}$ is a homogeneous function of degree $-r_3$. When $|l|-N < -n$, the term above is bounded by a constant multiple of 
        \begin{equation*}
        \begin{split}
            & \lambda^{-N} \epsilon^{-r_4}\int_{\Gamma_k \cap B(0, \epsilon)^{c}} |y|^{r-2N+|l|-r_1-r_3} \norm{\frac{\partial^{r_2}\eta}{\partial y_k^{r_2}}}_{L^{\infty}} \norm{\frac{\partial^{r_4} \alpha}{\partial y_k^{r_4}}}_{L^{\infty}} dy \\
            \lesssim &   \lambda^{-N} \epsilon^{-r_4} \epsilon^{r-2N+|l|-r_1-r_3+n} \norm{\eta}_{C^{r_2}}\norm{\alpha}_{C^{r_4}}\\
            \lesssim & \lambda^{-N} \epsilon^{|l|-2N+r-r_1-r_3-r_4+n} \norm{\eta}_{C^{r_2}}\norm{\alpha}_{C^{r_4}}.
        \end{split}
        \end{equation*}
        Then we obtain (\ref{eq_int_QmylOmegak}) by setting $\epsilon = \lambda^{-\frac{1}{2}}$. $A_k$ depends on $\norm{\eta}_{C^{N}}$.
    \item The proof for b) is similar to that for a). We use the same cone $\Gamma_k$ and functions $\Omega_k$, $\alpha$. It suffices to show
    \begin{equation}
    \label{eq_QmgOmegak}
        \left|\int_{\Gamma_k} e^{i \lambda Q_{m, n}(y)} g(y) \Omega_k(y) dy  \right| \leq B_{N, k} \lambda^{-N}\
    \end{equation}
    for a $B_{N, k}>0$. If $2\epsilon < \delta$, 
     \begin{equation*}
             \int_{\Gamma_{k}} e^{i \lambda Q_{m, n}(y)} g(y) \Omega_k(y) dy 
            =  \int_{\Gamma_k \cap B(0, \epsilon)^{c}} e^{i \lambda Q_{m, n}(y)} g(y) \Omega_k(y) [1-\alpha(\epsilon^{-1} y)] dy.   
        \end{equation*}
    Then, we apply integration by parts $N$ times to obtain
    \begin{equation}
    \label{eq_int_QmtDngOmegak}
          \int_{\Gamma_k \cap B(0, \epsilon)^{c}} e^{i \lambda Q_{m, n}(y)} (^tD)^N \left\{g(y) \Omega_k(y) [1-\alpha(\epsilon^{-1} y)] \right\}dy.
    \end{equation}
    After we apply (\ref{eq_tDkNg}) and the product rule of the derivative to expand (\ref{eq_int_QmtDngOmegak}), we obtain a summation of terms where a term, ignoring the constant, is
        \begin{equation*}
                \lambda^{-N}\int_{\Gamma_k} e^{i \lambda Q_{m, n}(y)} y_k^{r-2N}   \left[\frac{\partial^{r_1}}{\partial y_k^{r_1}}g(y)\right]\left[\frac{\partial^{r_2}}{\partial y_k^{r_2}}\Omega_k(y)\right]\left[\frac{\partial^{r_3}}{\partial y_k^{r_3}} [1-\alpha(\epsilon^{-1} y)] \right]dy
        \end{equation*}
        for $r_1, r_2, r_3 \geq 0$, $r_1+r_2+r_3 = r \leq N$. We note that $\frac{\partial^{r_2} \Omega_k}{\partial y_k^{r_2}}$ is a homogeneous function of degree $-r_2$. When $-N < -n$, the term above is bounded by a constant multiple of 
        \begin{equation*}
        \begin{split}
            & \lambda^{-N} \epsilon^{-r_3}\int_{\Gamma_k \cap B(0, \epsilon)^{c}} |y|^{r-2N-r_2} \norm{\frac{\partial^{r_1}g}{\partial y_k^{r_1}}}_{L^{\infty}} \norm{\frac{\partial^{r_3} \alpha}{\partial y_k^{r_3}}}_{L^{\infty}} dy \\
            \lesssim &   \lambda^{-N} \epsilon^{-r_3} \epsilon^{r-2N-r_2+n} \norm{g}_{C^{r_1}}\norm{\alpha}_{C^{r_3}}\\
            \lesssim & \lambda^{-N} \epsilon^{r-2N-r_2-r_3+n} \norm{g}_{C^{r_1}}\norm{\alpha}_{C^{r_3}}.
        \end{split}
        \end{equation*}
        Then, we obtain (\ref{eq_QmgOmegak}) by setting $\epsilon = \frac{\delta}{3}$. \qedhere
    \end{enumerate}
\end{proof}

\begin{theorem}\cite[Chapter 8, Proposition 6]{Stein1993}
\label{t_asym}
    \begin{enumerate}[a.]
        \item Let $Q_{m, n}(y)$ be defined as in (\ref{eq_Qm}). Let
        $$I_m(\lambda; \psi) : = \int_{\RR^n} e^{i \lambda Q_{m, n}(y)} \psi(y) dy,$$ where $\psi$ is supported in a small neighborhood of $0$, and $\lambda_0 > 1$. For $\lambda \geq \lambda_0$,
        \begin{equation}
        \label{eq_Im_main_d}
            \left|I_m(\lambda; \psi) - (-1)^{\frac{n-m}{2}}(\pi i)^{\frac{n}{2}} \psi(0) \lambda^{-\frac{n}{2}}\right| \leq D\lambda^{-\frac{n+1}{2}},
        \end{equation}
        where $D$ depends on $\lambda_0$, the size of $\text{spt } \psi$, and $\norm{\psi}_{C^{n+3}}$.
        
        \item Let $\phi, \psi \in C^{\infty}(\RR^{n})$. Suppose that $\phi$ has only one non-degnerate critical point $z_0$ in $\text{spt } \psi$ and $\phi(z_0)=0$. Let
        \begin{equation*}
            I(\lambda; \phi, \psi) : = \int_{\RR^n} e^{i \lambda \phi(z)} \psi(z) dz,
        \end{equation*}
        $\lambda_0 > 1$, and $c=\det \Delta \phi(z_0)$. For $\lambda \geq \lambda_0$, there exist $0\leq m \leq n$ and $D=D(\lambda_0, \phi, \psi)>0$, such that
        \begin{equation*}
            \left|I(\lambda; \phi, \psi) - (-1)^{\frac{n-m}{2}}(2\pi i)^{\frac{n}{2}} c^{-\frac{1}{2}} \psi(z_0) \lambda^{-\frac{n}{2}}\right| \leq D(\lambda_0, \phi, \psi)|\lambda|^{-\frac{n+1}{2}}.
        \end{equation*}
        The error $D(\lambda_0, \phi, \psi)$ depends on $\lambda_0$, $c$, the size of $\text{spt } \psi$, and $\norm{\phi}_{C^{n+6}}$, $\norm{\psi}_{C^{n+3}}$.
    \end{enumerate}
\end{theorem}

\begin{proof}

    \begin{enumerate}[a.]
        \item Step 1: We have
        \begin{equation*}
            \int_{-\infty}^{\infty} e^{i \lambda x^2} e^{-x^2} dx = (1-i\lambda)^{-\frac{1}{2}} \int_{-\infty}^{\infty} e^{-x^2} dx,
        \end{equation*}
        and $\int_{-\infty}^{\infty} e^{-x^2} dx = \sqrt{\pi}$. We can fix the principal branch of $z^{-\frac{1}{2}}$ in the complex plane slit along the negative real-axis. Therefore,
        \begin{equation}
        \label{eq_int_eny2}
        \begin{split}
            \int_{\RR^{n}} e^{i \lambda Q_{m, n}(y)} e^{-|y|^2} dy 
            = & \left(\prod_{j=1}^{m}\int_{\RR} e^{i \lambda y_j^2} e^{-y_j^2} dy_j\right)\left(\prod_{j=m+1}^{n}\int_{\RR} e^{-i \lambda y_j^2} e^{-y_j^2} dy_j\right) \\
            = & \pi^{\frac{n}{2}}(1-i\lambda)^{-\frac{m}{2}}(1+i\lambda)^{-\frac{n-m}{2}} \\
            = & \pi^{\frac{n}{2}} \lambda^{-\frac{n}{2}}(\lambda^{-1}-i)^{-\frac{m}{2}}(\lambda^{-1}+i)^{-\frac{n-m}{2}}.
        \end{split}
        \end{equation}
        
        We write the power series expansion of $f_m(w) = (w-i)^{-\frac{m}{2}}(w+i)^{-\frac{n-m}{2}}$ at $0$ as
        $\sum_{k=0}^{\infty} a_k w^k,$
        where $a_0 = (-1)^{\frac{n-m}{2}} i^{\frac{n}{2}}$. Let $\gamma$ be a line segment from $0$ to $\lambda^{-1}$ and $b_m =  \sup_{|w| = \lambda_0^{-1}} |f_m'(w)|$. Then for $\lambda \geq \lambda_0 > 1$, the error of approximation by the constant term is bounded by
        \begin{equation}
        \label{eq_fm_a0_d}
        \begin{split}
            |f_m(\lambda^{-1}) -  a_0| \leq & \left|\int_{\gamma} f_m'(w) dw \right|\\
            \leq & |\gamma| \sup_{|w| = \lambda^{-1}} |f_m'(w)| \\
            \leq & b_m \lambda^{-1}.
        \end{split} 
        \end{equation}
        The integral form of the Taylor remainder and the maximum modulus principle are used. Putting everything together, for $\lambda \geq \lambda_0$,
        \begin{equation}
        \label{eq_int_Qmny_main_d}
            \begin{aligned}
                 & \left|\int_{\RR^{n}} e^{i \lambda Q_{m, n}(y)} e^{-|y|^2} dy - a_0 \pi^{\frac{n}{2}} \lambda^{-\frac{n}{2}} \right| \\
             = & \left| \pi^{\frac{n}{2}} \lambda^{-\frac{n}{2}}(\lambda^{-1}-i)^{-\frac{m}{2}}(\lambda^{-1}+i)^{-\frac{n-m}{2}}  -  a_0 \pi^{\frac{n}{2}} \lambda^{-\frac{n}{2}}\right| & \text{ by (\ref{eq_int_eny2})}\\
             = & \pi^{\frac{n}{2}} \lambda^{-\frac{n}{2}} \left|  (\lambda^{-1}-i)^{-\frac{m}{2}}(\lambda^{-1}+i)^{-\frac{n-m}{2}}  -  a_0 \right| \\
            \leq & b_m \pi^{\frac{n}{2}} \lambda^{-\frac{n+2}{2}}  & \text{ by (\ref{eq_fm_a0_d})}.
            \end{aligned}
        \end{equation}

    Step 2: To obtain (\ref{eq_Im_main_d}), we write $e^{|y|^2}\psi(y) = \psi(0) + \sum_{y=1}^{n} y_j R_j(y)$, where $R_j \in C_0^{\infty}(\RR^n)$. Let $\bar{\psi} \in C^{\infty}_0(\RR^n)$, with $\bar{\psi}(y) = 1$ on the support of $\psi$. To apply the results from the previous steps, we write
    \begin{equation*}
        \begin{split}
            & \int_{\RR^{n}} e^{i \lambda Q_{m, n}(y)} \psi(y) dy \\ = & \int e^{i \lambda Q_{m, n}(y)} e^{-|y|^2} [e^{|y|^2} \psi(y)] \bar{\psi}(y) dy \\
            = & \int_{\RR^{n}} e^{i \lambda Q_{m, n}(y)} e^{-|y|^2} \left[\psi(0)+\sum_{j=1}^{n} y_j R_j(y)\right] \bar{\psi}(y) dy \\
            = & \int e^{i \lambda Q_{m, n}(y)} e^{-|y|^2} \psi(0) \bar{\psi}(y) dy + \sum_{j=1}^{n}\int e^{i \lambda Q_{m, n}(y)} y_j e^{-|y|^2} R_j(y) \bar{\psi}(y) dy \\
            = & \psi(0) \int e^{i \lambda Q_{m, n}(y)} e^{-|y|^2}  dy+ \psi(0) \int e^{i \lambda Q_{m, n}(y)} e^{-|y|^2}  [1- \bar{\psi}(y) ]dy \\
            & + \sum_{j=1}^{n} \int e^{i \lambda Q_{m, n}(y)} y_j e^{-|y|^2} R_j(y) \bar{\psi}(y) dy.
        \end{split}
    \end{equation*}
    
    Then we obtain (\ref{eq_Im_main_d}) by applying (\ref{eq_int_Qmny_main_d}) to the first integral, (\ref{eq_int_Qmg}) to the second integral, and (\ref{eq_int_Qmyl}) to each term in the summation. We note that the error $D$ depends on $\lambda_0$, the size of the support of $\psi$, and $\norm{\psi}_{C^{n+3}}$.
    \item 
     Let $V$ be a neighborhood of $z_0$ obtained from applying Morse's lemma (Lemma \ref{lem_morse}) to $\phi$, and we write $\psi = \psi_1+\psi_2$, where $\text{spt }\psi_1 \subset V \cap \text{spt }\psi$, and $\text{spt }\psi_2$ does not contain $z_0$. Then, we write
     $$ \int_{\RR^n} e^{i \lambda \phi(z)} \psi(z) dz =  \int_{\RR^n} e^{i \lambda \phi(z)} \psi_1(z) dz  +  \int_{\RR^n} e^{i \lambda \phi(z)} \psi_2(z) dz.$$
     \begin{itemize}
         \item By Morse's lemma (Lemma \ref{lem_morse}),
    there exists a diffeomorphism $\tau: V \to U$, where $V$ is a neighborhood of $z_0$ in the $z$-space, and $U$ is a neighborhood of $0$ in the $y$-space, such that $\tau(z_0)=0$ and 
    \begin{equation*}
        \phi(z) = Q_{m, n}(\tau(z)).
    \end{equation*} 
    By a change of variables $z=\tau^{-1}(y)$,
    \begin{equation*}
        \begin{split}
            \int_{\RR^n} e^{i \lambda \phi(z)} \psi_1(z) dz = & \int e^{i \lambda Q_{m, n}(y)}\psi_1(\tau^{-1}(y)) |\det {\bf J  }\tau^{-1}(y)|dy \\
            = & I_m(\lambda; (\psi_1 \circ \tau^{-1})\cdot |\det {\bf J  }\tau^{-1}|).
        \end{split}
    \end{equation*}
    We can apply part a) to the integral above. $D$ depends on $\lambda_0$, the size of support of $\psi_1 \circ \tau^{-1}$, and the $C^{n+3}$ norms of $(\psi_1 \circ \tau^{-1})\cdot |\det {\bf J  }\tau^{-1}|$. We note that the $L^{\infty}$ norm of the $k^{\text{th}}$ partial derivative ($0 \leq k \leq n+3$) of $(\psi_1 \circ \tau^{-1})\cdot |\det {\bf J  }\tau^{-1}|$ can be bounded by $C^{n+3}$ norms of $\psi$, $|\det {\bf J  }\tau|$, and $\inf\{|\det {\bf J  }\tau(z)| z \in \text{spt} \psi  \}$. We note that $|\det {\bf J  }\tau(0)|^2=\frac{c}{2^n}$ by (\ref{eq_JDJD}), and we can bound the error in terms of $\phi$ with the fact that $\norm{\det {\bf J  }\tau}_{C^{n+3}(V)}$  and $\inf\{|\det {\bf J  }\tau(z)| z \in \text{spt} \psi  \}$ depend on $\norm{\phi}_{C^{n+5}(V)}$, $c$, and the size of $\text{spt } \psi_1$ from the discussion of Morse's lemma in Section \ref{sec62}.
    \item For the second integral, we apply Proposition \ref{p_local}. \qedhere
     \end{itemize}  
    \end{enumerate}
\end{proof}

\bibliographystyle{plain} 
\bibliography{refs}

\Addresses

\end{document}